\documentclass[english]{article}
\usepackage[T1]{fontenc}
\usepackage[latin9]{inputenc}
\usepackage{amsmath}
\usepackage{amssymb}
\usepackage{babel}
\usepackage{amsthm}
\usepackage{graphicx}
\usepackage{tikz}
\usepackage[all]{xy}
\usepackage{bm}
\usepackage{comment}
\usetikzlibrary{calc,patterns,angles,quotes}
\usepackage[toc,page]{appendix}
\usepackage{hyperref}

\newtheorem{prop}{Proposition}
\newtheorem{theorem}{Theorem}
\newtheorem{lemma}{Lemma}

\newtheorem{cor}{Corollary}
\newtheorem{defi}{Definition}

\newcommand{\R}{\mathbb{R}}

\title{{Integrable Mechanical Billiards in Higher Dimensional Space Forms}}
\author{Airi Takeuchi, Lei Zhao}

\begin{document}

\maketitle
\begin{center}
	\it{Dedicated to Alain Chenciner on his 80th birthday.}
\end{center}

\begin{abstract}
{In this article, we consider mechanical billiard systems defined with Lagrange's integrable {extension} of Euler's two-center problems in the Euclidean space, on the sphere, and in the hyperbolic space of arbitrary dimension $n \ge 3$. In the three-dimensional Euclidean space, we show that the billiard systems with any finite combination of quadrics having two foci at the Kepler centers are integrable. The same holds for the projections of these systems on the three-dimensional sphere and the hyperbolic space by means of central projection. With the same approach, we also extend these results to the $n$-dimensional cases. } 

\end{abstract}

\section{Introduction}
{A natural mechanical system $(M,g,U)$ is defined on an $N$-dimensional Riemannian manifold $(M,g)$ {with} a smooth force function $U: M \to \R$. {{This} system determines the motion of a particle on a manifold, governed by the second-order \emph{Newton's equations}}
\begin{equation*}
\nabla_{\dot{q}}\dot{q}=  {\hbox{grad}_{g} U(q)}, \, q \in M,
\end{equation*}
in which {${\nabla}$} is the Levi-Civita connection associated to the Riemannian metric $g$ {and {$\hbox{grad}_{g}$} is the gradient with respect to $g$.}
}
{The kinetic energy of the system{, seen as a function defined on $TM$, is} given by 
\[
K(q,\dot{q}) = \frac{1}{2} g_q(\dot{q}, \dot{q}).
\]
The total energy {$E=K+V$} of the {system} is the combination of the kinetic energy $K$ and the potential $V:= -U$.  {{The total energy $E$} is always a first integral of the system, \emph{i.e.}, it is }preserved along the motion. 
}



{{We consider, {in addition}, the presence of a {(piecewise)} smooth{, codimension-1} reflection wall} $\mathcal{B}$ in a natural mechanical system $(M,g,U)$. {A particle moves in the system until it reflects elastically at $\mathcal{B}$, \emph{i.e.}, {the velocity does not change its norm, and its direction at the reflection point} is changed 
so that the angle of incidence {equals} the angle of reflection.} {As such, we obtain a billiard system defined with an underlying natural mechanical system. {We denote this system}  by $(M,g, U, \mathcal{B})$. Note that in the case {$\hbox{grad}_{g} U=0$}, we obtain the billiard system in the usual sense defined on $(M, g)$. When in addition, $M$ is $\R^{2}$ equipped with the standard Euclidean metric, then we obtain the usual Birkhoff billiards in the plane. Their dynamics depends only on the shape of the reflection walls, and have been extensively studied. The dynamics of a general mechanical billiard system, on the other hand, also depends crucially on the underlying mechanical system. 
}

{{For} a mechanical billiard system $(M,g,U, \mathcal{B})$, a quantity} is called {a} first integral {if} it is a first integral of the underlying system $(M,g,U)$
and, in addition, does not change its value at the reflections at $\mathcal{B}$. 

{It is clear from the construction}  that the total energy {$E$} is preserved under the reflections and {is always} a first integral of the system. {The dynamics of the system depends in a crucial way on whether other first integrals exist. Recent progresses on the Birkhoff-Poritsky Conjecture \cite{Kaloshin-Sorrentino}, \cite{Bialy-Mironov} show that for Birkhoff billiards in the plane, only highly special billiard domains could admit additional first integrals. These results somehow hint that mechanical billiards admitting additional first integrals should be rare and highly special as well.}

{The billiard system $(M,g,U, \mathcal{B})$ is called integrable if {there exist $n=\dim M$ independent first integrals which are mutually in involution}. Integrable mechanical billiard systems should be rare when $\dim M =2$ since the existence of a first integral in addition to the energy is in general not expected. When $\dim M \ge 3$, integrable mechanical billiards should have a higher rarity since we need more independent first integrals in addition to the energy to exist. }

{
The trajectories of the classical free (Birkhoff) billiards are completely determined by the trajectories of billiard mapping which sends the pair of a point of reflection and the direction of a velocity to the pair of consecutive ones. {The situation is completely analogous for  mechanical billiard systems, and the billiard mapping thus obtained is symplectic with respect to a proper symplectic form.} 
The discrete version of the Arnold-Liouville theorem \cite{Veselov} tells that for {such} integrable symplectic maps,
the phase space of the mapping is foliated by invariant tori, and {the} system evolutes linearly on each torus in appropriate coordinates.} {So the dynamics of an integrable mechanical billiard system is mostly ordered, in contrast to the dynamics of non-integrable (mechanical) billiard systems, which can be much more chaotic.}

{{It was probably L. Boltzmann who was the first to consider billiard systems in a potential field, or what we call mechanical billiards.} In \cite{Boltzmann}, Boltzmann considered the mechanical billiard system with the underlying central force problem in $\R^2$ defined with {the} force function 
\[
V_{\alpha,\beta } = \frac{\alpha}{2r} - \frac{\beta}{2r^2}, \quad \alpha, \beta \in \R
\]
in which $r$ is the distance of the particle from the origin $O$, and with a line not passing the origin as {the} reflection wall. {In this paper, he thought {he had} shown that such a billiard system is ergodic.} {G. Gallavotti \cite{Gallavotti} realized that there was a gap in the argument and conjectured, based on the numerics of I. Jauslin, that the case $\beta=0$ should be integrable, which he subsequently proved with I. Jauslin \cite{Gallavotti-Jauslin}, by constructing an additional first integral in addition to the energy of the system. The integrable dynamics of this {mechanical billiard system} was analyzed by Felder \cite{Felder}, in which a more direct proof was presented. In \cite{zhao2021}, a conceptual interpretation of the additional first integral of Gallavotti-Jauslin, and thus yet a different proof of this integrability, was given with the theory of projective dynamics. Moreover, the analysis in \cite{Felder} shows that for sufficiently small $|\beta|$, the system remains non-ergodic by application of KAM theory. Thus, in order for Boltzmann's ergodic assertion to be true, $|\beta|$ has to be sufficiently large with respect to $|\alpha|$. Some numerical investigations can be found in {\cite{ATThesis}}. 
}

{Euler's two-center problem is a classical integrable system \cite{Euler}. Lagrange modified the system by adding a Hooke center at the middle point of the two Kepler centers, and the resulting system remains integrable \cite{Lagrange}. 
These systems can be defined analogously on the sphere and in the hyperbolic plane, by taking the proper form of potentials generalizing the Kepler and Hooke potentials, based on the works of Serret  \cite{Serret} for the sphere and Killing \cite{Killing} for {general n-dimensional space forms}. As a further generalization of the line of the investigation initiated by Gallavotti, Jauslin, and Felder, in \cite{Takeuchi-Zhao1} , \cite{Takeuchi-Zhao2}, we established the integrability of mechanical billiard systems within the two-center problem/Lagrange problem and with any finite combination of confocal conic sections focused at the Kepler centers as walls of reflection in the plane, on the sphere, and in the hyperbolic plane. }

{The method we used in \cite{Takeuchi-Zhao2} is based on projective billiards and projective dynamics. The projective {viewpoint} has been used for the study of integrability of geodesics \cite{Tabachnikov2}, \cite{Tabachnikov 1999}, \cite{Topalov-Matveev}. In \cite{Tabachnikov 1999}, \cite{Topalov-Matveev}, based on the notion of projective equivalence, the geodesic flow on the ellipsoid is shown to be integrable. The application of the projective {viewpoint} to free billiard systems gives rise to the study of projective billiards \cite{Tabachnikov3}, \cite{Tabachnikov 2002}.
The projective viewpoint for mechanical systems is developed by A. Albouy to what is now called projective dynamics \cite{Albouy1}. An elegant application of this theory obtained by him is the conceptual proof of the integrability of the Lagrange problem \cite{Albouy2}. The method we {used} in \cite{Takeuchi-Zhao2} {combined} these two aspects.
}

}

 {The purpose of this article is to discuss further generalizations of these integrable mechanical billiards in a space of constant {sectional curvature} of dimension at least 3}. {The discussion will be made on the simply connected ones, \emph{i.e.,} the Euclidean space, the $n$-dimensional sphere, and the $n$-dimensional hyperbolic space. Examples of integrable mechanical billiards on non-simply connected space forms are obtained from integrable mechanical billiards on their universal covers by restrictions. Since the systems we are dealing with are typically not invariant under the automorphism group, such a restriction procedure typically leads to {incomplete,} singular systems.}

{As walls of reflection we consider {what we call \emph{confocal  quadrics}}. In the Euclidean space, these are the non-degenerate quadric hypersurfaces having two foci in the sense of optics, namely, each ray emanating from one of them passes through the other after reflection at the quadric. Moreover, we assume the two Kepler centers are the foci, with the axis through the centers as a symmetric axis\footnote{For these quadric hypersurfaces, this actually follows from the requirement of having two foci.}. In $\R^{3}$, these are {spheroids (ellipsoids of revolution) or circular hyperboloids (hyperboloids of revolution) of two sheets}  focused at the two Kepler centers.  
On the sphere and in the hyperbolic space, these are defined via the central projections. {See Definition \ref{def: symmetric_quadrics_S3} and Definition \ref{def: symmetric_quadrics_H3}.} 
Notice that our definition of confocal quadrics is more restrictive than the terminology which has been used in a more common way, for example in \cite{Odehnal}.

{Our result is summarized in the following theorem: }

\begin{theorem} \label{thm: int_main}
{The Lagrange billiard systems in the Euclidean space, on the sphere, and in the hyperbolic space of any dimension $n \ge 3$, with any finite combination of confocal quadrics as walls of reflection, are integrable.}
\end{theorem}
{A simple observation is that the Lagrange billiard systems with any open subsets of such finite combination of confocal quadrics as reflection walls are again integrable. For example, we can consider a sheet of the circular hyperboloid of two sheets, or even any open subset of it. }

{{ In the sequel, we shall present proofs as well as various subcases of this theorem.} Note that for the Kepler problem, some limiting integrable cases can be obtained: For example, in $\R^{3}$, a circular paraboloid focused at the Kepler center also defines an integrable Kepler billard system. Another limiting case, namely the mechanical billiard with a circular paraboloid reflection wall in a uniform gravitational force field, has been analyzed in \cite{Jaud}.} Besides, it is classically known that the billiard system with any ellipsoid reflection wall centered at a Hooke center, not necessarily possesses an axis of symmetry, defines an integrable Hooke billard system \cite{Kozlov-Treshchev}, which is not a special case of our theorem. We shall not discuss these limiting or special cases further in detail. }



{The approach in this paper is to use projective dynamics to treat the higher-dimensional case. Indeed with this method, we always obtain a first integral in addition to the energy. The other first integrals are obtained from the symmetric condition in Theorem \ref{thm: int_main}. }

 {{We organize this article as follows:}
 In Section \ref{sec: central_projection}, we recall the theory of projective dynamics and the projective correspondence between force fields defined on the three-dimensional hemisphere and the three-dimensional Euclidean space. In Section \ref{sec: main_resutls}, we {establish the} integrability of the Lagrange billiard defined in the three-dimensional space and on the three-dimensional {sphere}. 
In Section \ref{sec: hyperboloid}, we prove the {integrability} for the hyperbolic case.  {Finally,} in Section \ref{sec: higher_dimensional}, we generalize our results to the higher dimensional case. }

{





\section{Central Projection and Projective Systems}
\label{sec: central_projection}

In the four-dimensional Euclidean space $\R^4$, we consider the three-dimensional subspace $W=\{ (x,y,z, -1) \} \subset \R^4$ and the three-dimensional unit sphere 
$$\mathcal{S}^3:=\{ (q_1,q_2,q_3,q_4)\in \R^4 \mid q_1^2 +q_2^2+q_3^2+q_4^2=1   \} .$$
The central projection from the origin $O=(0,0,0,0)$ of $\R^{4}$ projects the points of the south-hemisphere $\mathcal{S}_{SH}:=\{(q_1,q_2,q_3,q_4) \in \mathcal{S}^4 \mid q_4 <0 \}$ to the points in the three-dimensional space  $W$. We equip $\mathcal{S}^3$ and $\mathcal{S}_{SH}$ with their induced, round metrics from $\R^{4}$, {while for $W$ we allow a further affine change {from its induced metric}.} We denote the Euclidean inner product and the Euclidean norm of $\R^{4}$ {respectively} by $\langle \cdot , \cdot \rangle$ and $\|\cdot\|$. {In the theory of projective dynamics \cite{Albouy2}, \cite{Albouy1},} a force field $F_{W}$ on $W$ is projected to a force field $F_{S}$ on $\mathcal{S}_{SH}$ by the push-forward of the central projection with a factor of time change uniquely determined by the projection. Since the direction of the projection is orthogonal to $\mathcal{S}_{SH}$, by the computation which follows, {the trajectories of these systems are therefore also related by the central projection up to a time parametrization.}

We start {with} the force field $F_{W}$ in $W$ and deduce the corresponding force field $F_{S}$ on $\mathcal{S}_{SH}$. The equations of motion of the system in $W$ is
\begin{equation}
\label{eq: newton_eq_W}
\ddot{\tilde{q}}=F_{W}(\tilde{q}),
\end{equation}
where $\tilde{q}\in W$.
Let $q \in \mathcal{S}_{SH}$ be projected to $\tilde{q} \in W$ by the central projection:
$$
q = \| \tilde{q} \|^{-1} \tilde{q}.
$$
From this, we compute
$$
\dot{q}  = \| \tilde{q} \|^{-2} (\dot{\tilde{q}}  \| \tilde{q} \| - \langle \nabla \| \tilde{q}\| , \dot{\tilde{q}} \rangle \tilde{q}),
$$
where we write $\dot \,:=\dfrac{d}{dt}$ the time derivative. 
We now take a new time variable $\tau$ for the system on $\mathcal{S}_{SH}$ such that 
\begin{equation}
	\label{eq: time_parameter_change}
	{\dfrac{d}{d \tau}= \| \tilde{q} \|^{2}\dfrac{d}{d t}}
\end{equation}
 and write as $' :=\dfrac{d}{d\tau}$.
We thus have
$$
q'  =  (\dot{\tilde{q}}  \| \tilde{q} \| - \langle \nabla \| \tilde{q}\| , \dot{\tilde{q}} \rangle \tilde{q}).
$$
and
\begin{equation}
	\label{eq: second_derivative}
	q''= \| \tilde{q} \|^{2} (\ddot{\tilde{q}}  \| \tilde{q} \| - (\langle \nabla \| \tilde{q}\| , \ddot{\tilde{q}} \rangle+\langle d(\nabla \| \tilde{q}\|)/dt , \dot{\tilde{q}} \rangle) \tilde{q}).
\end{equation}

After pluging \eqref{eq: newton_eq_W} into above, we have
\begin{equation}\label{eq: second derivative correspondence}
	q''=\| \tilde{q} \|^{2} (F_{W}(\tilde{q})  \| \tilde{q} \|- \lambda(\tilde{q}, \dot{\tilde{q}}, \ddot{\tilde{q}}) \tilde{q}),
\end{equation}
where $\lambda(\tilde{q}, \dot{\tilde{q}}, \ddot{\tilde{q}})=\langle \nabla \| \tilde{q}\| , \ddot{\tilde{q}} \rangle+\langle {d(\nabla \| \tilde{q}\|)/dt} , \dot{\tilde{q}} \rangle$.

We observe that the first term of the right-hand side of this equation depends only on $\tilde{q}\in W$ and therefore, depends only on $q \in \mathcal{S}_{SH}$ by the central projection, {while the second term is radial: It is normal to $\mathcal{S}_{SH}$ and vanishes after being projected to the tangent space}. Projecting both sides of this equation to the tangent space $T_{q}  \mathcal{S}_{SH}$ we get the equations of motion on $S_{SH}$ given in the form
$$\nabla_{q'} q'=F_{S}(q).$$

In general, if we start with a natural mechanical system in $W$, the corresponding projected system defined on $\mathcal{S}_{SH}$ is not necessarily derived from a potential. {When this indeed holds true, then we call the initial natural mechanical system \emph{projective}:}

{
\begin{defi}
	\label{def: projective}
	A natural mechanical system in $W$ is called \emph{projective} if the projected system on $\mathcal{S}_{SH}$ is also a natural mechanical system, \emph{i.e.,} the projected force field $F_S$ on $\mathcal{S}_{SH}$ is derived from some force function. 
 \end{defi}
}

{Note that while it is certainly possible to extend this definition to allow more general situations, this definition is sufficient for our purpose.}

We now consider the following three-dimensional natural mechanical systems and {explore} their projective properties.



\begin{defi}
	\label{def: Lagrange_R3}
	The Lagrange problem in the three-dimensional Euclidean space $(\R^{3}, \| \cdot \|)$ is the system 
	\begin{equation}\label{sys: spatial_Lagrange}
		(\R^{3}, \|\cdot\|, m_{1} /\|q-Z_{1}\| + m_{2} /\|q-Z_{2}\|+f \|q-(Z_{1}+Z_{2})/2\|^{2}),
	\end{equation}
	with $m_1, m_2, f \in \R$, which is the superposition of two Kepler problems and a Hooke problem, with the Kepler centers $Z_1$ and $Z_2$ placed symmetrically with respect to the Hooke center. 
	{As special cases, we have:}
	{\begin{itemize}
			\item {$f= 0$:  The Hooke center is ineffective, and the motion of the particle is governed by the collective attraction or repulsion from the two Kepler centers. This system is the spatial two-center problem.}
			\item  {$m_1 = f=0 $ or $m_2 = f= 0$. The motion of the particle is governed only by the attraction or repulsion of a Kepler center, which is the spatial Kepler problem.}
			\item {When $m_1 = m_2 = 0$, the particle moves under the influence of a Hooke center. The system is the spatial Hooke problem}.
		\end{itemize}
	}
\end{defi}

	They have analogous (hemi)spherical systems on $\mathcal{S}_{SH}$ and on $\mathcal{S}^3$ as their projections \cite{Serret}, \cite{Appel}.  



\begin{defi}
	\label{def: Lagrange_def_hemisph3}
	The {hemi}spherical Lagrange problem on $(\mathcal{S}_{SH}, \| \cdot \|)$ with its two Kepler centers at $Z_1, Z_2\in \mathcal{S}_{SH}$ and its Hooke center $Z_0 \in \mathcal{S}_{SH}$ in the middle of two Kepler centers $Z_1$ and $Z_2$ is the system 
	\begin{equation}\label{sys: sph_Lagrange}
		(\mathcal{S}_{SH}, \|\cdot\|, m_1\cot \theta_{Z_1} +m_2 \cot \theta_{Z_2} + f\tan^2 \theta_{Z_0}),
	\end{equation}
	where $\theta_{Z_1}= \angle q O Z_1, \theta_{Z_2}= \angle q O Z_2, \theta_{Z_0} = \angle q O Z_0$  with $m_1,m_2, f\in \R$.
	{As special cases, we have:}
	{\begin{itemize}
			\item {$f= 0$:  The Hooke center is ineffective, and the motion of the particle is governed by the collective attraction or repulsion from the two Kepler centers. This system is the hemispherical two-center problem.}
			\item  {$m_1 = f=0 $ or $m_2 = f= 0$. The motion of the particle is governed only by the attraction or repulsion of a Kepler center, which is the hemispherical Kepler problem.}
			\item {When $m_1 = m_2 = 0$, the particle moves under the influence of a Hooke center. The system is the hemispherical Hooke problem}.
		\end{itemize}
	}
\end{defi}
{These hemispherical systems can also be defined on the whole sphere by setting the antipodal points of repulsive/attracting centers as attracting/repulsive centers with sign-changed mass factors.} 

\begin{defi}
	\label{def: Lagrange_ded_sph3}
	{The spherical Lagrange problem on $(\mathcal{S}^3, \| \cdot \|)$ with its Kepler centers at $Z_1, Z_2\in \mathcal{S}_{SH}$ and at their antipodal points $Z_1', Z_2' \in \mathcal{S}^3$ and its Hooke centers at $Z_0 \in \mathcal{S}_{SH}$, which is located in the middle of two Kepler centers $Z_1$ and $Z_2$, and at its antipodal point $Z_0' \in \mathcal{S}^3$ is the system 
		\begin{equation}\label{sys: whole_sph_Lagrange}
			(\mathcal{S}^{3}, \|\cdot\|, m_1\cot \theta_{Z_1} -m_1\cot \theta_{Z_1'} +m_2 \cot \theta_{Z_2} -m_2 \cot \theta_{Z_2'}+ f\tan^2 \theta_{Z_0}-f\tan^2 \theta_{Z'_0}),
		\end{equation}
		where $\theta_{Z_1}= \angle q O Z_1, \theta_{Z_2}= \angle q O Z_2, \theta_{Z_0} = \angle q O Z_0$, ${\theta_{Z'_1}= \angle q O Z'_1, \theta_{Z'_2}= \angle q O Z'_2,}$ ${\theta_{Z'_0} = \angle q O Z'_0}$  with $m_1,m_2, f\in \R$.} 
	{As well as hemispherical cases, we have the following systems as subsystems of the Lagrange problem:}
	{\begin{itemize}
			\item {$f= 0$: This system is the spherical two-center problem.}\footnote{Actually, the spherical two-center problem has four centers appearing in two pairs. The same terminology has been chosen for the consistence with other cases.}
			\item  {$m_1 = f=0 $ or $m_2 = f= 0$: This system is the spherical Kepler problem.}
			\item { $m_1 = m_2 = 0$: This system is the spherical Hooke problem}.
		\end{itemize}
	}
\end{defi}

{Indeed, the above systems are projective in the sense of Definition \ref{def: projective} as we will see in the following Lemma \ref{lem: projective}.}

{We fix the centers in $W$ as $$\tilde{Z}_1 = (a,0,0,-1), \tilde{Z}_2 = (-a,0,0,-1), \tilde{Z}_0 = (0,0,0,-1) \in W$$
	for some $a \in \R$, and equip the space $W$ with the norm $\| \cdot \|_a$ {defined} as
	\[
	\| \tilde{q} \|_a:=\sqrt{\frac{x^2}{1+a^2}+y^2+z^2}
	\]
	for $\tilde{q} = (x,y,z, -1) \in W$. } 

\begin{lemma}
	\label{lem: projective}
	The three-dimensional Lagrange problem in $(W, \| \cdot\|_a)$ {with the Hooke center at $(0,0,0,-1)$} is projective.
\end{lemma}
\begin{proof}

	
	{{We consider}  the spatial Lagrange problem in $(W, \|\cdot \|_a )$ with the centers $\tilde{Z}_1, \tilde{Z}_2, \tilde{Z}_0$.}
	The force field $F_W$ of this system is given by 
	\[
	F_W(\tilde{q}) = - m_1 \| \tilde{q}  - \tilde{Z}_1 \|_a^{-3}(\tilde{q}  - \tilde{Z}_1) - m_2 \| \tilde{q}  - \tilde{Z}_2 \|_a^{-3}(\tilde{q}  - \tilde{Z}_2) + 2f (\tilde{q} - \tilde{Z}_0) 
	\]
	We now plug this into the right-hand side of \eqref{eq: second derivative correspondence} and compute its projection
	to the tangent space of $T_q \mathcal{S}_{SH}$. {Observe} that we only need to consider the first term of {the right-hand side of} \eqref{eq: second derivative correspondence} since the second term vanishes after the projection to $T_q \mathcal{S}_{SH}$.

	Let points $q \in \mathcal{S}_{SH}$ and $ \tilde{q} \in W$ {be related via the central projection}:
	\[
	q=(q_1,q_2,q_3,q_4) \mapsto  \tilde{q}=(x,y,z,-1),
	\]
	where $x = -\frac{q_1}{q_4}, y = \frac{q_2}{q_4}, z= -\frac{q_3}{q_4}$.
	
	The corresponding force field on $\mathcal{S}_{SH}$ is determined by the projection of $\| \tilde{q} \|^3 F_W(\tilde{q})$ to $T_{q}\mathcal{S}_{SH}$.

	\begin{figure}
		\begin{tikzpicture}
			\begin{scope}[scale= 1.0]
				\draw[-,>=stealth, thick] (-3.5,-2)--(2.5,-2)node[right]{$\ell_i$}; 
				\coordinate (l) at (2.5,-2);
				\coordinate (O) at (0, 0);
				\fill[black] (O) circle (0.05);
				\coordinate (G) at (0, -2);
				\fill[black] (G) circle (0.05);
				\coordinate (Zt) at (-2, -2);
				\fill[black] (Zt) circle (0.05);
				\coordinate (qt) at (-3, -2);
				\fill[black] (qt) circle (0.05);
				
				\draw (O) circle [radius=1.5];
				
				\coordinate (Z) at (-1.06, -1.06);
				\fill[black] (Z) circle (0.05) node[below]{$Z_i$};
				\coordinate (q) at (-1.24, -0.83);
				\fill[black] (q) circle (0.05) node[left]{$q_i$};
				
				\coordinate (P) at ($(G)!.3cm!(l)!.3cm!90:(l)$);
				\draw[thick, red] ($(G)!(P)!(l)$)--(P);
				\draw[thick, red] ($(G)!(P)!(O)$)--(P);
				
				\draw [-] (O)node[right]{O} -- (G) node[below]{$G_i$};
				\draw (O) -- (Zt) node[below]{$\tilde{Z}_i$};
				\draw (O) -- (qt) node[below]{$\tilde{q}_i$};
			\end{scope}

			\begin{scope}[xshift = 7.0cm, scale=1.0]
				\coordinate (Zt) at (-2,-1);
				\coordinate (-Zt) at (2,-1);
				\fill[black] (Zt) circle (0.05) node[below]{$\tilde{Z}_i$};
				\coordinate (Z0) at (0.25,-1);
				\fill[black] (Z0) circle (0.05) node[below]{$\tilde{Z}_0$};
				\draw[-,>=stealth] (-3,-1)--(2,-1);
				\coordinate (G) at (-0.5,-2.0);
				\fill[black] (G) circle (0.05) node[below]{$G_i$};
				\draw[-,>=stealth, thick] (-3,-0.3)--(1.5,-3.4)node[right]{$\ell_i$} ;
				\coordinate (l) at (1.5,-3.4);
				\draw[-,>=stealth] (Z0)--(G);
				\coordinate (O) at (0.25,1);
				\fill[black] (O) circle (0.05) node[right]{$O$};
				\draw[-,>=stealth] (O)--(Z0);
				\draw[-,>=stealth] (O)--(G);
				
				\coordinate (P1) at ($(Z0)!.2cm!(-Zt)!.2cm!90:(-Zt)$);
				\draw[thick, red] ($(Z0)!(P1)!(-Zt)$)--(P1);
				\draw[thick, red] ($(Z0)!(P1)!(O)$)--(P1);
				
				\coordinate (P2) at ($(G)!.2cm!(l)!.2cm!90:(l)$);
				\draw[thick, red] ($(G)!(P2)!(l)$)--(P2);
				\draw[thick, red] ($(G)!(P2)!(Z0)$)--(P2);
			\end{scope}
		\end{tikzpicture}
		\caption{{Sectional Views Containing $\ell_i$ and $O$.}}
		\label{fig: section_views}
	\end{figure}
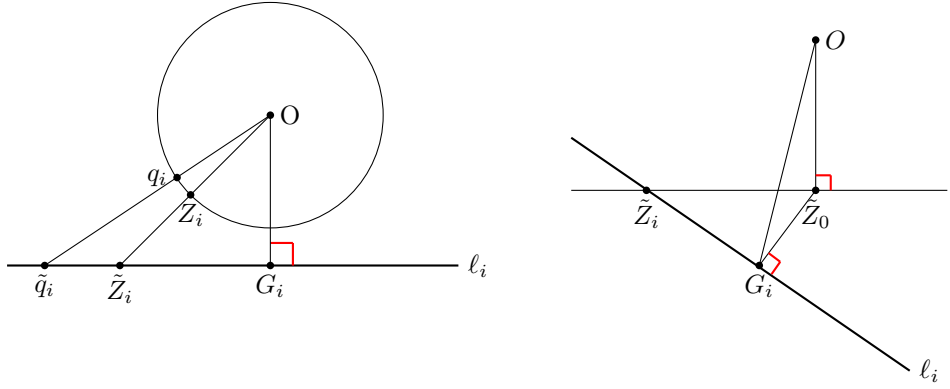

	Let $\ell_i$ be the line passing two points $\tilde{Z}_i$ and $\tilde{q}$, {and let $G_i $ be the point in $\ell_i$ such that $OG_i$ is perpendicular to  $\ell_i$, for each $i = 1,2$.} {See Figure \ref{fig: section_views} for their geometrical illustrations.}
	The projection of $\| \tilde{q} \|^3 F_W(\tilde{q})$  is computed as
	\scriptsize
	\[
	\| \tilde{q} \|^3\left(- m_1 \| \tilde{q}  - \tilde{Z}_1 \|_a^{-3}(\tilde{q}\  - \tilde{Z}_1)\cdot \frac{\|G_1\|}{\| \tilde{q}\|} - m_2 \| \tilde{q}  - \tilde{Z}_2 \|_a^{-3}(\tilde{q}  - \tilde{Z}_2)\cdot \frac{\|G_2\|}{\| \tilde{q}\|} + 2f (\tilde{q} - \tilde{Z}_0)\cdot \frac{\|Z_0\|}{\| \tilde{q}\|}\right).
	\]
	\normalsize
	Using the following equations:
	\[
	\|\tilde{q} - \tilde{Z}_1\|_a = \|\tilde{q} - \tilde{Z}_1\| \cdot \frac{\|G_1\|}{\sqrt{1+a^2}}, \quad	\|\tilde{q} - \tilde{Z}_2\|_a = \|\tilde{q} - \tilde{Z}_2\| \cdot \frac{\|G_2\|}{\sqrt{1+a^2}},
	\]
	the norm of the projected force is computed as 
	\begin{align*}
		&\frac{| \hat{m}_1 | \|\tilde{q}\|^2 \| \tilde{Z}_1\|^2}{\|\tilde{q} - \tilde{Z}_1\|^{2} \| G_1\|^{2}} + \frac{| \hat{m}_2 | \|\tilde{q}\|^2 \| \tilde{Z}_2\|^2}{\|\tilde{q} - \tilde{Z}_2\|^{2} \| G_2\|^{2}} + 2 |f| \|\tilde{q}\|^2 \| \tilde{q} - \tilde{Z}_0 \|  \\
		=& | \hat{m}_1 | \sin^{-2} \theta_{\tilde{Z}_1} + | \hat{m}_2 | \sin^{-2} \theta_{\tilde{Z}_2} + 2 |f| \frac{\sin \theta_{\tilde{Z}_0}}{\cos^3 \theta_{\tilde{Z}_0}},
	\end{align*}
	where $\hat{m}_1 = m_1 \sqrt{1 + a^2}, \hat{m}_2 = m_2\sqrt{1 + a^2}$ and $\theta_{\tilde{Z}_1}=\angle \tilde{Z}_1 O \tilde{q}$, $\theta_{\tilde{Z}_2}=\angle \tilde{Z}_2O \tilde{q}$, and $\theta_{\tilde{Z}_0}=\angle \tilde{Z}_0O\tilde{q}$.
	Therefore, the corresponding system on $\mathcal{S}_{SH}$ is the three-dimensional spherical Lagrange problem with two Kepler centers at
	$$Z_1=\left(\frac{a}{\sqrt{1+a^2}},0,0, -\frac{
		1}{\sqrt{1+a^2}}\right),Z_2=\left(-\frac{a}{\sqrt{1+a^2}},0,0, -\frac{
		1}{\sqrt{1+a^2}}\right)\in \mathcal{S}_{SH}$$ and a Hooke center at $Z_0= (0,0,0,-1)\in \mathcal{S}_{SH}$ defined with the force function 
	\[
	U(q)=\frac{\hat{m}_1}{\tan \theta_{Z_1}} + \frac{\hat{m}_2}{\tan \theta_{Z_2}} + f \tan^2 \theta_{Z_0}, 
	\]
	where  $\theta_{Z_1}=\angle Z_1 O q$, $\theta_{Z_2}=\angle Z_2Oq$, and $\theta_{Z_0}=\angle Z_0Oq$.

\end{proof}

{We now make a short digression on Poisson brackets. For two smooth functions $F, G$ on $T^{*} M$, their Poisson bracket $\{ F,G  \}$ 
	is defined as
}
\[
\{ F,G  \} =\dfrac{d}{dt}\Bigr|_{t=0} F(\gamma^t_{G}) =\omega (X_F,X_G),
\]
{in which $\omega$ is the canonical symplectic form on $T^{*} M$.}

{If $\{ F,G  \}$=0, then the functions $F$ and $G$ are \emph{in involution}. Clearly in this case, $F$ is preserved along the flow of $X_{G}$ and vice versa.}
}

{We now pull-back the Poisson bracket to $TM$ as follows:}
{For the Riemannian manifold $(M, g)$, we can identify $TM$ and $T^{*} M$ via the isomorphism
\[
\rho: T_{q}M  \to  T^*_qM, \quad  v \mapsto p:=g(v, \cdot ).
\]
}
{For any smooth functions $F,G$ on $T^*M$, we define {their pull-backs} on $TM$ as
\[
\overset{\circ}F:= F \circ \rho, \quad \overset{\circ}{G} := G \circ \rho,
\]
{We equally pull back the Hamiltonian flow $\gamma_G^t$ to $TM$ via }
\[
 {\rho^{-1} \circ \gamma_G^t \circ \rho.}
\]
{Now we define the bracket} 
\[
\{\overset{\circ}{F},\overset{\circ}{G}  \}_{TM}:= \dfrac{d}{dt}\Bigr|_{t=0} \overset{\circ}{F}\circ  \rho^{-1} \circ \gamma_G^t \circ \rho.
\]
{We prefer to pull back the Poisson bracket to $TM$, since this facilitates our discussions based on projective dynamics in this article.}
The following easy proposition states the equivalence between {the commutativity with respect to the } {pulled-back} bracket $\{\cdot, \cdot \}_{TM}$ and the {usual} Poisson bracket on $T^{^*} M$.  
}
\begin{prop}
\label{prop: Poisson}
{For any smooth functions $F,G$ on $T^*M$, we have
	\[
	\{\overset{\circ}{F},\overset{\circ}{G} \}_{TM} =0 \Leftrightarrow \{F,G  \}=0.
	\]
}
\end{prop}
\begin{proof}
{From the definition, it follows that 
	\[
	\begin{split}
	\{\overset{\circ}{F},\overset{\circ}G  \}_{TM} &= \dfrac{d}{dt}\Bigr|_{t=0} \overset{\circ}{F}\circ  \rho^{-1} \circ \gamma_G^t \circ \rho \\
	&= \dfrac{d}{dt}\Bigr|_{t=0} {F}\circ \rho \circ \rho^{-1} \circ \gamma_G^t \circ \rho \\
	&=  \dfrac{d}{dt}\Bigr|_{t=0} F\circ \gamma_{G}^t \circ \rho \\
	&= \{F,G\}\circ \rho.
	\end{split}
	\]
}
\end{proof}

{From now on we shall abbreviate the notation and write $\{\overset{\circ}{F},\overset{\circ}{G}  \}$ instead of $\{\overset{\circ}{F}, \overset{\circ}{G}  \}_{TM}$ for two smooth functions $\overset{\circ}{F}, \overset{\circ}{G}$ on $TM$.}

The projective property shown in Lemma \ref{lem: projective} provides one additional first integral for each spherical and spatial Lagrange problem. To describe them explicitly, without loss of generality, we assume that the spatial Lagrange problem in $(W, \| \cdot \|_a)$ has its centers at 
$$\tilde{Z}_1 = (a,0,0,-1), \tilde{Z}_2 = (-a,0,0,-1), \tilde{Z}_0 = (0,0,0,-1),$$ and the spherical Lagrange problem on $\mathcal{S}_{SH}$ has its center at 
\scriptsize
\[
Z_1=\left(\frac{a}{\sqrt{1+a^2}},0,0, -\frac{
	1}{\sqrt{1+a^2}}\right),Z_2=\left(-\frac{a}{\sqrt{1+a^2}},0,0, -\frac{
	1}{\sqrt{1+a^2}}\right), Z_0 = (0,0,0,-1).
\]
\normalsize
{The energy of the spherical Lagrange problem on $\mathcal{S}_{SH}$ given by 
\normalsize
\begin{equation}
\label{eq: E_sph}
E_{sph}:= \frac{\| {q'}\| ^2}{2} - \frac{\hat{m}_1}{\tan \theta_{Z_1}} - \frac{\hat{m}_2}{\tan \theta_{Z_2}} - f \tan^2 \theta_{Z_0}, 
\end{equation}
\normalsize
induces an additional first integral for the spatial Lagrange problem in $W$ which  independent of energy for the spatial system
\begin{equation}
\label{eq: E_sp}
E_{sp}:=  \frac{\| \dot{\tilde{q}}\|_a ^2}{2} - \frac{m_1}{\| \tilde{q} - \tilde{Z}_1 \|_a}
- \frac{m_2}{\| \tilde{q} - \tilde{Z}_2 \|_a} - f \| \tilde{q} - \tilde{Z}_0 \|_a^2,
\end{equation}
{where $m_1= \frac{\hat{m}_1}{\sqrt{1+a^2}}, m_2= \frac{\hat{m}_2}{\sqrt{1+a^2}} $.}
{	Taking the time change given by \eqref{eq: time_parameter_change} into account, the kinetic part of the spherical energy $E_{sph}$ can be written in the chart of $W$ as
	\small
	\begin{align*}
	\tilde{K}_{sph} &:= \frac{(y^2 + z^2 + 1)\dot{x}^2 + (x^2 + z^2 + 1)\dot{y}^2 + (x^2 + y^2 + 1)\dot{z}^2 -2 xy \dot{x} \dot{y} - 2 xz \dot{x} \dot{z} - 2 yz \dot{y}\dot{z}}{2}\\
	& = \frac{\dot{x}^2 +  \dot{y}^2 + \dot{z}^2 + (x \dot{y} - y \dot{x})^2 + (y \dot{z} - z \dot{y})^2 + (z \dot{x} - x \dot{z})^2}{2}.
	\end{align*}
	\normalsize
	By also rewriting the potential part of $E_{sph}$ in the same chart, we obtain the following expression for the spherical energy in the chart of $W$:
	\scriptsize
	\begin{equation}
	\label{eq: E_sph_proj}
	\begin{split}
	\tilde{E}_{sph} := &\tilde{K}_{sph} - f (x^2 + y^2 + z^2)- \frac{\hat{m}_1( 1 + ax)}{(x- a)^2 + (1 + a^2)(y^2 + z^2 )}
	- \frac{\hat{m}_2 (1 - ax)}{(x+ a)^2 + (1 + a^2)(y^2 + z^2 )}.
	\end{split}
	\end{equation}
	\normalsize
	Also, the angular momentum 
	\begin{equation}
	\label{eq: L_yz}
	L_{yz} := y \dot{z} - z \dot{y}
	\end{equation}
	is conserved the spatial Lagrange problem, since the three centers $\tilde{Z}_1, \tilde{Z}_2$, and $\tilde{Z}_0$ lie in the $x$-axis of $W$. 
}

{Similarly, the hemispherical Lagrange problem on $\mathcal{S}_{SH}$ has the two first integrals besides its own energy $E_{sph}$, namely, the projections of the spatial energy $E_{sp}$ and the angular momentum $L_{yz}$. {Using external coordinates for $\mathcal{S}^3$, these} can be represented as 
	\small
	\begin{equation}
	\label{eq: E_sp_proj}
	\begin{split}
	\hat{E}_{sp} := &\frac{(q_4 q_1'- q_1 q_4')^2}{2  (1+a^2)} + \frac{(q_4 q_2'- q_2 q_4')^2}{2 } + \frac{(q_4 q_3'- q_3 q_4')^2}{2  }  - f \left(\frac{q_1^2}{q_4^2(1+a^2)} + \frac{q_2^2}{q_4^2} + \frac{q_3^2}{q_4^2} \right)\\
	&- m_1\left(\frac{(-q_1/q_4-a)^2}{1+a^2} + \frac{q_2^2+q_3^2}{q_4^2}  \right)^{-\frac{1}{2}}-m_2\left(\frac{(-q_1/q_4+a)^2}{1+a^2} + \frac{q_2^2+q_3^2}{q_4^2}  \right)^{-\frac{1}{2}}
	\end{split}
	\end{equation}
	\normalsize
	and 
	\begin{equation}
	\label{eq: L_yz_proj}
	\hat{L}_{yz} := q_3 q_2' - q_2 q_3'.
	\end{equation}
}

{
	As a corollary of Lemma \ref{lem: projective}, {we obtain} the integrability of these two corresponded systems related by the central projection.}

\begin{cor}
	\label{cor: integrability_dim_3}
	{The Lagrange problems in $\R^3$, on $\mathcal{S}_{SH}$, and on $\mathcal{S}^3$ are integrable. More specifically, for the spatial Lagrange problem in $(W, \| \cdot\|_a)$ with its centers at $Z_0 =(0,0,0,-1), \tilde{Z}_1=(a, 0,0,-1), \tilde{Z}_2= (-a,0,0,-1)$, its own energy $E_{sp}$ \eqref{eq: E_sp}, the projected spherical energy $\tilde{E}_{sph}$ \eqref{eq: E_sph_proj}, and the angular momentum $L_{yz}$ \eqref{eq: L_yz}
	are three independent first integrals which are mutually in involution. For the spherical system on $\mathcal{S}_{SH}$ with its center at 
	\small
	\[
	Z_0 =(0,0,0,-1), Z_1=\left(\dfrac{a}{1+a^2}, 0,0,-\dfrac{1}{1+a^2}\right), Z_2= \left(-\dfrac{a}{1+a^2},0,0,-\dfrac{1}{1+a^2}\right),
	\]
	\normalsize
	 its energy $E_{sph}$ \eqref{eq: E_sph}, the projected spatial energy $\hat{E}_{sp}$ \eqref{eq: E_sp_proj}, and the angular momentum $\hat{L}_{yz}$ \eqref{eq: L_yz_proj} are three independent first integrals which are mutually in involution. 
	}
\end{cor}
\begin{proof}

	{
		To check that {$E_{sp}$, $\tilde{E}_{sph}$, and $L_{yz}$ are functionally independent}, we will show {that the rank of the Jacobian }
		\begin{equation}
		\label{eq: Jacobian}
		\begin{pmatrix}
		\frac{\partial  E_{sp}}{\partial  x} & \frac{\partial  E_{sp}}{\partial  y}  & \frac{\partial  E_{sp}}{\partial  z} & \frac{\partial  E_{sp}}{\partial  \dot{x}} & \frac{\partial  E_{sp}}{\partial  \dot{y}} & \frac{\partial  E_{sp}}{\partial  \dot{z}} \\
		\frac{\partial  \tilde{E}_{sph}}{\partial  x} & \frac{\partial  \tilde{E}_{sph}}{\partial  y}  & \frac{\partial  \tilde{E}_{sph}}{\partial  z} & \frac{\partial  \tilde{E}_{sph}}{\partial  \dot{x}} & \frac{\partial  \tilde{E}_{sph}}{\partial  \dot{y}} & \frac{\partial  \tilde{E}_{sph}}{\partial  \dot{z}} \\
		\frac{\partial  L_{yz}}{\partial  x} & \frac{\partial  L_{yz}}{\partial  y}  & \frac{\partial  L_{yz}}{\partial  z} & \frac{\partial  E_{sp}}{\partial  \dot{x}} & \frac{\partial  L_{yz}}{\partial  \dot{y}} & \frac{\partial  L_{yz}}{\partial  \dot{z}} 
		\end{pmatrix}
		\end{equation}
		{is 3. This is established by observing that} the submatrix 
		\small
		\begin{align*}
		&
		\begin{pmatrix}
		\frac{\partial  E_{sp}}{\partial  \dot{x}} & \frac{\partial  E_{sp}}{\partial  \dot{y}} & \frac{\partial  E_{sp}}{\partial  \dot{z}} \\
		\frac{\partial  \tilde{E}_{sph}}{\partial  \dot{x}} & \frac{\partial  \tilde{E}_{sph}}{\partial  \dot{y}} & \frac{\partial  \tilde{E}_{sph}}{\partial  \dot{z}} \\
		\frac{\partial  E_{sp}}{\partial  \dot{x}} & \frac{\partial  L_{yz}}{\partial  \dot{y}} & \frac{\partial  L_{yz}}{\partial  \dot{z}} 
		\end{pmatrix}
		= 
		\begin{pmatrix}
		\dot{x}/(1+a^2)  & \dot{y} & \dot{z} \\
		\nu_x \dot{x} - xy \dot{y} - xz \dot{z} & \nu_y \dot{y} - xy \dot{x} - yz \dot{z} & \nu_z \dot{z} - xz \dot{x} - yz \dot{y}  \\
		0 & -z  & y
		\end{pmatrix}
		\end{align*}
		{has already} rank 3,
		\normalsize
		where {$$\nu_x:=y^2 + z^2 + 1, \nu_y:=x^2 + z^2 + 1, \nu_z:=x^2 + y^2 + 1.$$}

		We now check that first integrals $E_{sp}$, $\tilde{E}_{sph}$, and $L_{yz}$ are in involution. Since they all are conserved along the flow {of $E_{sp}$}, we have 
		\[
		\{E_{sp}, \tilde{E}_{sph}  \}= 0, \quad \{ E_{sp}, L_{yz}  \} = 0.
		\]
	}
	{Note that we use the abbreviated notation for the {induced Poisson} bracket {for} functions on $T W$. {See} Proposition \ref{prop: Poisson}.}
	
	 { The rotation in $TW$ generated by $L_{yz}$ {sends the point $(x,y,z,\dot{x}, \dot{y},\dot{z})$ to $$(x, \cos \phi \cdot  y - \sin \phi \cdot z, \sin\phi \cdot y {+} \cos \phi \cdot z, \dot{x}, \cos \phi \cdot  \dot{y} - \sin \phi \cdot \dot{z}, \sin\phi \cdot \dot{y} {+} \cos \phi \cdot \dot{z}) $$ for any $\phi \in  \R / 2 \pi$. We see that $\tilde{E}_{sph}$} is invariant under this rotation
		\normalsize
		which implies 
		\[
		\{ \tilde{E}_{sph}, L_{yz}  \} =  0.
		\]
	}

{For the hemispherical system,} the first integrals $E_{sph}, \hat{E}_{sp}$, and $\hat{L}_{yz}$ are functionally independent. Indeed, the central projection from $\mathcal{S}_{SH}$ to $W$ is a diffeomorphism, and it induces an isomorphism between $T \mathcal{S}_{SH}$ and $T W$, thus the rank of Jacobian
	\[
	\begin{pmatrix}
	\nabla^T \hat{E}_{sp} \\
	\nabla^T E_{sph}\\
	\nabla^T \hat{L}_{yz}
	\end{pmatrix}
	\]
	is equal to the rank of Jacobian given in \eqref{eq: Jacobian}. Therefore, the functional independence follows from the previous computation.  
	
	We are now going to see that they are also in involution.
	We denote the Poisson bracket defined on functions on $T^* \mathcal{S}^3$ by $\{\cdot, \cdot  \}_{\mathcal{S}^3}$ and {as notation,} we use the same bracket for functions on $T \mathcal{S}^3$ due to {Proposition \ref{prop: Poisson}}.
	
	Since the planar energy $\hat{E}_{sp}$ and the angular momentum $\hat{L}_{yz}$ are first integrals of the system, we already have
	\[
	\{E_{sph}, \hat{E}_{sp}  \}_{\mathcal{S}^3}= 0, \quad\{  E_{sph}, \hat{L}_{yz}  \}_{\mathcal{S}^3}=0.
	\]
	
	We are going to check that $\hat{E}_{sp}$ and $\hat{L}_{yz}$ are also in involution.
	The angular momentum $\hat{L}_{yz}$ generates the {$S^{1}$-family of} rotations {on $T\mathcal{S}^3$ sending the point $(q_2,q_3,q_2',q_3')$ to} the point
	\small
	\[
	{ (\cos \phi \cdot  q_2 - \sin \phi \cdot q_3, \sin\phi \cdot q_2 {+} \cos \phi \cdot q_3,\cos \phi \cdot  q_2' - \sin \phi \cdot q_3', \sin\phi \cdot q_2' {+} \cos \phi \cdot q_3').}
	\]
	\normalsize
	{The function $\hat{E}_{sp}$ is invariant under these rotations.}
	Therefore, we obtain
	\[
	\{ \hat{E}_{sp}, \hat{L}_{yz} \}_{\mathcal{S}^3}= 0.
	\] 
	{{The kinetic part of the} projected spatial energy $\hat{E}_{sp}$:
	\[
	\frac{(q_4 q_1'- q_1 q_4')^2}{2  (1+a^2)} + \frac{(q_4 q_2'- q_2 q_4')^2}{2 } + \frac{(q_4 q_3'- q_3 q_4')^2}{2  } 
	\]
	can be analytically extended to the whole sphere $\mathcal{S}^3$, and the potential part 
	\scriptsize
	\[
	- \frac{m_1}{\sqrt{\frac{(-q_1/q_4-a)^2}{1+a^2} + \frac{q_2^2+q_3^2}{q_4^2}}} -\frac{m_2}{\sqrt{\frac{(-q_1/q_4+a)^2}{1+a^2} + \frac{q_2^2+q_3^2}{q_4^2}}} - f \left(\frac{q_1^2}{q_4^2(1+a^2)} + \frac{q_2^2+q_3^2}{q_4^2} \right)
	\]
	\normalsize
	can {also} be extended to the whole sphere, outside of the singularities which are Kepler centers and their antipodal points and the horizontal equator 
	$${\{(q_1,q_2,q_3,q_4)\in \mathcal{S}^3 \mid q_4 = 0\}}.$$}
\end{proof}
\section{{Projective Billiards Systems and Main Results in the Three-Dimensional Case}}
\label{sec: main_resutls}

{We now add 
{reflection walls} and consider the corresponding billiard systems. Note that the Lagrange problem has two Kepler centers and a Hooke center. {Not all of them are assumed to be effective. The line of centers is the line that passes through all effective centers.}} {We shall show via projective dynamics the integrability of general Lagrangian billiards {in $\R^3$ and on $\mathcal{S}^3$} with confocal quadrics focused at the Kepler centers.} As opposed to the more studied cases of free billiards and Hooke billiards, these seem to be the only integrable case. 

{We fix the Kepler centers at $\tilde{Z}_1, \tilde{Z}_2 \in \R^3$. We assume that they are distinct: $\tilde{Z}_{1} \neq \tilde{Z}_{2}$.} {Without loss of generality, we set $\tilde{Z}_{1}=(-a, 0, 0), \tilde{Z}_{2}=(a, 0, 0)$. We equip $\R^{3}$ with a norm, not necessarily the standard Euclidean one. 
}

{With foci we mean two points such that the quadric is defined as the locus of points for which either the sum or the absolute value of the difference of the distances to those two points is constant.}
{Quadrics with two foci are either spheroids or circular hyperboloids of two sheets.}

\begin{defi}
	A quadric in $\R^3$ is called confocal if it has two foci at $\tilde{Z}_{1}, \tilde{Z}_{2}$.

\end{defi}

{After normalization, we may write {a confocal quadric}
\[
\frac{x^2}{A^2} + \frac{y^2}{B^2} + \frac{z^2}{B^2} = 1, \quad {A,B>0} 
\]
for a spheroid, and 
\[
\frac{x^2}{A^2} - \frac{y^2}{B^2} - \frac{z^2}{B^2} = 1, \quad {A,B>0} 
\]
for a hyperboloid of two sheets. We understand that they are both focused at $\tilde{Z}_{1}=(-a, 0, 0), \tilde{Z}_{2}=(a, 0, 0)$. 
We remind that a relationship between the parameters $A, B$ and $a$ can be made explicit with the help of the norm on $\R^{3}$.
}

{We now take two Kepler centers $Z_1$ and $Z_2$ on $\mathcal{S}^{3}$. Without loss of generality, we assume that they lie in $\mathcal{S}_{SH}$ and are the projections of $\tilde{Z}_1$ and $\tilde{Z}_2$:
\[
Z_1= \left( \frac{a}{\sqrt{1+a^2}}, 0,0, -\frac{1}{\sqrt{1+a^2}} \right), Z_2= \left( -\frac{a}{1+a^2}, 0,0, -\frac{1}{\sqrt{1+a^2}} \right).
\]
We equip $\mathcal{S}^3$ with the round metric induced from the Euclidean metric of $\R^4$.} {For a quadric in $\mathcal{S}^3$, we define its foci analogously as for quadrics in $\R^3$, but now using the spherical distances.}
\begin{defi}
	\label{def: symmetric_quadrics_S3}
	{A spherical quadric with two foci on $\mathcal{S}_{SH}$ or on $\mathcal{S}^3$ is called confocal if it is focused at $Z_1$ and $Z_2$.}
\end{defi}

{A confocal quadric on $\mathcal{S}_{SH}$ has its center at the Hooke center $Z_0= (0,0,0,-1)$ and is, after normalization, given by 
\begin{displaymath}
	\left\{
	\begin{array}{l}
		q_{1}^{2}+q_{2}^{2}+q_{3}^{2}+q_{4}^{2}=1, \quad q_4 <0 \\
		\frac{q_1^2 }{A^2} + \frac{q_2^2}{B^2} + \frac{q_3^2}{B^2} - q_4^2 = 0, \quad {A,B>0}
	\end{array}
	\right.
\end{displaymath}
for a spheroid, and by
\begin{displaymath}
	\left\{
	\begin{array}{l}
		q_{1}^{2}+q_{2}^{2}+q_{3}^{2}+q_{4}^{2}=1, \quad q_4 <0 \\
		\frac{q_1^2 }{A^2} - \frac{q_2^2}{B^2} - \frac{q_3^2}{B^2} - q_4^2 = 0, \quad {A,B>0}
	\end{array}
	\right.
\end{displaymath}
for a circular hyperboloid of two sheets.}

{We {now} consider elastic reflections against a reflection wall $\tilde{\mathcal{B}}$ in $\R^3$ {and against} the projection $\mathcal{B}$ on $\mathcal{S}_{SH}$ of  $\tilde{\mathcal{B}}$ by the central projection. {Note that in general, projection and reflection do not commute, \emph{i.e.,} the projection of the reflection velocity is not equal in general to the reflection of the projected velocity. We define} the projective correspondence between their law of reflections at $\tilde{\mathcal{B}}$ and $\mathcal{B}$ if projection and reflection do commute:}
\begin{defi}
	{Let {the} reflection walls $\tilde{\mathcal{B}}$ in $\R^3$ and $\mathcal{B}$ on $\mathcal{S}_{SH}$ be centrally projected to each other.  {We say that $\tilde{\mathcal{B}}$ and $\mathcal{B}$ are in {projective correspondence}} if and only if the incoming vector and the outgoing vector of the elastic reflection at any point $\tilde{q} \in \mathcal{\tilde{B}}$ are projected to the incoming vector and the outgoing vector, respectively, of the elastic reflection at the projected point $q \in \mathcal{B}$ of $\tilde{q}$.}
\end{defi}

\begin{lemma}
\label{lem: reflection_law}
{Let $\mathcal{\tilde{B}}$ be a confocal quadric reflection wall in $(W, \| \cdot \|_a)$
with foci at $\tilde{F}_1=(a,0,0)$ and  $\tilde{F}_2=(-a,0,0)$. Also, let $\mathcal{B}$ be a spherical quadric on $\mathcal{S}_{SH}$ given as the projection of $\mathcal{\tilde{B}}$. Then $\mathcal{B}$ is a spherical confocal quadric, and it has its foci at the projections of $\tilde{F}_1$ and $\tilde{F}_2$ on $\mathcal{S}_{SH}$. {Moreover, $\mathcal{\tilde{B}}$ and $\mathcal{B}$ are in projective correspondence.}

}
\end{lemma}

\begin{proof}
	We first consider the case that $\tilde{\mathcal{B}}$ {is a spatial spheroid} and $\mathcal{B}$ {is a spherical spheroid}.
	Set $\tilde{\mathcal{B}}$ as the spheroid given by 
	\begin{equation}
	\label{eq: ellipsoid}
	\frac{x^2}{A^2} + \frac{y^2}{B^2} + \frac{z^2}{B^2} = 1, \quad A> B.
	\end{equation}
	such that 
	\[
	1+a^2 = \frac{A^2+1}{B^2+1}.
	\]
	This {spheroid} has its foci at $\tilde{F}_1$ and $\tilde{F}_2$.  Indeed, if we set its foci {at} 
	$$\tilde{F}_1 = (c, 0,0,), \tilde{F}_2 = (-c,0,0),$$
	 then the value $c$ can be computed {via}
	\[
	\frac{c^2}{1+a^2} =\frac{A^2}{1+a^2} - B^2 \Leftrightarrow c = \pm a.
	\]
	{
		We now {describe} the projection {$\mathcal{B}$}  of  {$\tilde{\mathcal{B}}$} on {$\mathcal{S}_{SH}$}. {By the central projection, the quadric $\mathcal{B}$ on {$\mathcal{S}_{SH}$}  is given {by} the equation}
		\begin{equation}
		\label{eq: quadric_sphere}
		\frac{q_1^2}{A^2} + \frac{q_2^2}{B^2} + \frac{q_3^2}{B^2} - q_4^2 = 0, \quad {q_4<0}
		\end{equation}
		where $(q_1,q_2,q_3,q_4) \in {\mathcal{S}^3} {\subset \mathbb{R}^{4}}$. 
		
			{
			By setting $q_3= 0$, the above equation can be written into:
			\[
			\frac{q_1^2}{A^2} + \frac{q_2^2}{B^2} - q_4^2 = 0, \quad q_4<0.
			\]
			This is the equation of spherical ellipsoid on the hemisphere $$\mathcal{S}_{SH}^2= \{ q \in \mathcal{S}^3 \mid q_3 = 0, q_4<0 \}$$ with foci at $$F_1=(a/\sqrt{1+a^2}, 0,0,-1/\sqrt{1+a^2})$$ and $$F_2=(-a/\sqrt{1+a^2}, 0,0,-1/\sqrt{1+a^2}).$$ From the symmetry of \eqref{eq: quadric_sphere}, the {spherical} distances from the two points $F_1$ and $F_2$ do not depend on the value of $q_3$, therefore the quadric $\mathcal{B}$ is a spherical {spheroid} on $\mathcal{S}_{SH}$ with foci at $F_1$ and $F_2$.
		}
		{We now set} $q_1^2 = \kappa$ and {treat $\kappa$ as a parameter}. {From equation \eqref{eq: quadric_sphere} and the equation $q_1^2 + q_2^2 + q_3^2 + q_4^2 = 1$, we thus obtained the circle in the $(q_{2}, q_{3})$-plane:}
		\begin{equation}
		\label{eq: circular_q2q3}
		q_2^2 + q_3^2 = \frac{1 - (A^{-2} + 1)\kappa}{B^{-2} + 1}
		\end{equation}
		Consequently we have 
		\[
		{q_{4}} = - \sqrt{\frac{B^{-2} + (A^{-2} - B^{-2})\kappa}{B^{-2}+ 1}}.
		\]
		{Equivalently,}
		\[
		(1-A^{-2}B^2)q_1^2 + (1+ B^2)q_4^2 = 1, \quad q_4 <0.
		\]
		{Thus the} projection of $\mathcal{B}$ to {the $(q_1, q_4)$}-plane is the {``lower half ellipse.''}  
		
		We now consider the $S^1$-action {by rotations in the $(q_{2}, q_{3})$-plane on the sphere $\mathcal{S}^3$: } 
		\begin{equation}
		\label{eq: S1_action}
		\mathcal{S}^{3}  \ni (q_1, q_2,q_3, q_4) \mapsto (q_1, \cos \phi \cdot q_2 - \sin \phi \cdot q_3, \sin \phi \cdot q_2 + \cos \phi \cdot q_3, q_4) \in \mathcal{S}^{3}
		\end{equation}
		for $\phi \in {\mathbb{R} / 2 \pi \mathbb{Z}}$. It is clear that the hemisphere $S_{SH}$ is invariant under this action. 
		From equation \eqref{eq: quadric_sphere}, the spherical {spheroid} $\mathcal{B}$ is invariant under this $S^1$-action.
		Using $q_1^2 + q_2^2 + q_3^2 + q_4^2 = 1$, we eliminate $q_4$ from \eqref{eq: quadric_sphere}
		\[
		({A^{-2}} + 1)q_1^2 + ({B^{-2}} + 1)q_2^2 + ({B^{-2}} + 1)q_3^2 = 1. 
		\]
	}
	
	{The axis of symmetry of $\tilde{\mathcal{B}}$ is the $q_1$-axis in $W$, {which} is projected to the half equator given by 
		\[
		q_4 = {- |q_1|}, \quad q_2=q_3 = 0
		\]
		on $\mathcal{S}_{SH}$.}
	
	To check that {they} are projectively corresponded, it suffices to show that a normal vector to  $\mathcal{B}$ at $q \in \mathcal{B}$ is projected to again a normal vector to $\tilde{\mathcal{B}}$ at the projection $\tilde{q} \in \tilde{\mathcal{B}}$. {We explain this. Let $v$ and $w$ be the velocity vectors on $\mathcal{S}_{SH}$ and in $W$, respectively, and assume they are projectively correspond {to each other}. The velocity vectors before and after a reflection at the wall $\mathcal{B}$ are given by 
	\[
	v = v_t + v_n, v'= v_t - v_n,
	\]
	where $v_t$ is the tangent component of $v$ to $\mathcal{B}$ and $v_n$ is the normal component. 
	{Similarly,} we write the velocities before and after the reflection at $\mathcal{\tilde{B}}$ as
	\[
	w = w_t + w_n, w'= w_t - w_n.
	\]
	If $v_t$ projects to $w_t$ and $v_n$ projects to $w_n$, then $v'$ projects to $w'$. Since $\mathcal{B}$ and $\mathcal{\tilde{B}}$ are projected to each other, the tangent velocities $v_t$ and $w_t$ are also projected to each other. Hence, we are left to show that the normal velocities $v_n$ and $w_n$ are projected to each other. } 
	
	 Indeed, the normal vector to $\mathcal{B}$ at $q=(q_1,q_2,q_3,q_4) \in \mathcal{B} \subset \mathcal{S}^3$ given by 
	\[
	N:=\left(\frac{2 q_1}{A^2}, \frac{2 q_2}{B^2},\frac{2 q_3}{B^2},-2q_4\right)^T
	\]
	is projected to the vector $n$ in $W$ given by 
	\[
	n:=
	\begin{pmatrix}
	-\dfrac{1}{q_4} & 0 & 0 & \dfrac{q_1}{q_4^2}\\
	0 & -\dfrac{1}{q_4} & 0 & \dfrac{q_2}{q_4^2} \\
	0 & 0 & -\dfrac{1}{q_4} & \dfrac{q_2}{q_4^2}
	\end{pmatrix}
	\cdot 
	N
	= 
	\begin{pmatrix}
	-\dfrac{2 q_1(A^2+1)}{q_4 A^2}\\
	-\dfrac{2 q_2(B^2+1)}{q_4 B^2}\\
	-\dfrac{2 q_3(B^2+1)}{q_4 B^2}\\
	\end{pmatrix}
	\] 
	{To see that the projected vector $n$ is normal to $\mathcal{\tilde{B}}$ at projected point $$\tilde{q}=(x,y,z)= (-q_1/q_4, -q_2/q_4, -q_3/q_4),$$ we now compute the gradient of 
	\[
	 f= \frac{x^2}{A^2} + \frac{y^2}{B^2} + \frac{z^2}{B^2} -1
	\]
	at $\tilde{q}$ with respect to the metric $g$ of $W${:}
	\begin{align*}
	\nabla_g f &= \left(-\frac{2q_1}{q_4A^2}\cdot (1+a^2), - \frac{2q_2}{q_4B^2},- \frac{2q_3}{q_4B^2} \right) \\
	& = \left(-\frac{2q_1}{q_4A^2}\cdot \frac{A^2+1}{B^2+1}, - \frac{2q_2}{q_4B^2},- \frac{2q_3}{q_4B^2} \right) \\
	&= n \cdot \frac{1}{B^2 +1}
	\end{align*}
	which is {indeed} normal to the {spheroid} $\mathcal{\tilde{B}}$ at {$\tilde{q}$.} 
	}
	
	Similarly, if we start with a circular hyperboloid of two sheets $\tilde{\mathcal{B}}$ in $W$ given by 
	\begin{equation}
	\label{eq; hyperbolid}
	\frac{x^2}{A^2} - \frac{y^2}{B^2} - \frac{z^2}{B^2} = 1,
	\end{equation}
	with foci at {$\tilde{F}_1 = (a,0,0)$ and $\tilde{F}_2 = (-a,0,0)$}, which leads to the equation
	\[
	1+ a^2= \frac{A^2+ 1}{-B^2 + 1}.
	\]
	The projection $\mathcal{B}$ of the circular hyperboloid of two sheets $\tilde{\mathcal{B}}$ on $\mathcal{S}_{SH}$ is given by 
	\begin{equation}
	\label{eq: quadric_sphere_hyp}
	\frac{q_1^2}{A^2} - \frac{q_2^2}{B^2} - \frac{q_3^2}{B^2} - q_4^2 = 0, \quad {q_4<0}
	\end{equation}
	where $(q_1,q_2,q_3,q_4) \in {\mathcal{S}^3} {\subset \mathbb{R}^{4}}$. The {rest of the} argument directly follows from the elliptic case by changing $B$ to be a purely imaginary number.
\end{proof}


\begin{defi}
	{An} $N$-dimensional mechanical billiard system $(M,g,U ,\mathcal{B})$ is integrable if there exists $N$ independent first integrals $\{F_n\}_{n=1}^N$ for $(M,g,U)${, which are} {mutually} in involution, {and are invariant under reflections at $\mathcal{B}$, \emph{i.e.}
	\[
	F_n (q,v) = F_n(q,v'), \quad  {1 \leq n \leq N}
	\]
	where $v$ and $v'$ are {respectively the} incoming and outgoing velocities of the reflection at $q \in \mathcal{B}$.}
\end{defi}

{
}
\begin{theorem}
	\label{thm: spatial_Lagrange}
	The mechanical billiard systems defined in the 3-dimensional Euclidean space with the Lagrange problem 
	\[
	{	(\R^{3}, \|\cdot\|, m_{1} /\|q-Z_{1}\| + m_{2} /\|q-Z_{2}\|+f \|q-(Z_{1}+Z_{2})/2\|^{2}),}
	\]
	and with any finite combination of confocal quadrics with foci at the two Kepler centers as reflection wall, are integrable.
	
	{As subcases, we obtain the following results:
		The mechanical billiard problems defined in the three-dimensional Euclidean space 
	\begin{itemize}
		\item with the spatial  two-center problem and with any finite combination of confocal quadrics with foci at the two centers as reflection walls, 
		\item with the spatial Kepler problem and with any finite combination of confocal quadrics with one of the foci at the Kepler center as reflection walls,
		\item  with the spatial Hooke problem and with any finite combination of confocal quadrics centered at the Hooke  center as reflection walls, 
	\end{itemize}	
	are integrable. {The first integrals are $E_{sp}$ \eqref{eq: E_sp}, $\tilde{E}_{sph}$ \eqref{eq: E_sph_proj}, and $L_{yz}$ \eqref{eq: L_yz}.} 
	}
\end{theorem} 

{The consequence with only the Hooke problem is more restrictive as compared to the known result \cite{Kozlov-Treshchev}, in which the condition that the centered quadrics be symmetric with respect to a line containing the Hooke center can be relaxed. } 

\begin{proof}

	From  {Lemma \ref{lem: reflection_law}}, the spherical energy $\tilde{E}_{sph}$ is conserved under reflections at $\tilde{\mathcal{B}}$. This is also true for the angular momentum $L_{yz}$.
	{We will verify this directly. {Suppose that $\tilde{\mathcal{B}}$ is the spheroid given by \eqref{eq: ellipsoid}}. Due to the symmetry of $\tilde{\mathcal{B}}$, we can assume $z = 0$ without loss of generality. The outgoing vector after the reflection at  $\tilde{q} = (x,y,0) \in \mathcal{\tilde{B}}$ can be computed by using the normal vector {$$n:= \left(\frac{x(1+a^2)}{A^2}, \frac{y}{B^2},0\right)$$} to $\tilde{\mathcal{B}}$ at $\tilde{q}$ as
	{
	\begin{align*}
	v:=(v_x,v_y,v_z)&= \dot{\tilde{q}} - \frac{\langle \dot{\tilde{q}}, n \rangle_a}{\|n \|_a}n \\
	&= \left( \dot{x}  - \frac{(\frac{ \dot{x} x}{A^2}+\frac{ \dot{y} y}{B^2}  )(1+a^2)x}{A^2\sqrt{\frac{x^2}{A^4}+\frac{y^2}{B^4}}},  \dot{y}  - \frac{(\frac{ \dot{x} x}{A^2}+\frac{ \dot{y} y}{B^2} )y}{B^2\sqrt{\frac{x^2}{A^4}+\frac{y^2}{B^4}}},
	\dot{z}  
	\right)
	\end{align*} 
	}
	Then the angular momentum $L_{yz}$ after the reflection is 
	\begin{equation*}
	v_z y 
	= \dot{z} y,
	\end{equation*}
	{which is equal to $L_{yz}$ before the reflection.}
    }{Thus, the angular momentum $L_{yz}$ is conserved under reflections at $\mathcal{B}$.}
	
	{{We} have two first integrals {which are independent of the energy} and mutually in involution, {namely, the spherical energy $E_{sph}$ and the angular momentum $L_{yz}$}, for the billiard system defined with the Lagrange problem in $W$ with the reflection wall $\mathcal{B}$. {Therefore}, the billiard system is integrable.}
	
	{For {circular hyperboloids of two sheets}, the conservation of the spherical energy and the angular momentum $L_{yz}$ under reflections at $\mathcal{B}$
	directly follows from the computations for the case of {spheroids} by changing $B$ to be a purely imaginary number. }
	

\end{proof}

{Similarly, we obtain the following integrable spherical Lagrange billiard systems.}
\begin{theorem}
	{The spherical mechanical billiard systems defined on $\mathcal{S}^3$ with the Lagrange problem and with any {finite} combination  of {spherical} confocal quadrics with foci at two non-antipodal Kepler centers as reflection walls, are integrable.}
	
	{As subcases, we obtain the following results:
		The mechanical billiard problems defined on $\mathcal{S}^3$
		\begin{itemize}
			\item with the spherical two-center problem and with any {finite} combination of spherical confocal quadrics with foci at two non-antipodal centers as reflection walls, 
			\item with the spherical Kepler problem and with any {finite} combination of spherical confocal quadrics with one of the foci at a Kepler center as reflection walls,
			\item  with the spherical Hooke problem and with any {finite} combination of spherical confocal quadrics centered at a Hooke center as reflection walls, 
		\end{itemize}	
		are integrable. {The first integrals are $E_{sph}$ \eqref{eq: E_sph}, $\hat{E}_{sp}$ \eqref{eq: E_sp}, and the angular momentum $\hat{L}_{yz}$ \eqref{eq: L_yz_proj}.}
	}
\end{theorem}
  \begin{proof}
  	
 	
 	{The conservation of $\hat{E}_{sp}$ and $\hat{L}_{yz}$ under reflections {are directly deduced} from the same argument for the conservation of $E_{sp}$ and $L_{yz}$ in $W$, since $\hat{E}_{sp}$ and $\hat{L}_{yz}$ are the projections of $E_{sp}$ and $L_{yz}$, respectively, and in addition, $\mathcal{B}$ and $\tilde{\mathcal{B}}$ {are in projective correspondence}.
 	}

  \end{proof}

\section{{Integrable Lagrange Billiards in the 3-D Hyperbolic Space }}
\label{sec: hyperboloid}

We now discuss the integrability of {the analogously-defined Lagrange billiards} in the {three-dimensional} hyperbolic space. {For this purpose, we shall realize the three-dimensional hyperbolic space via a sheet of a two-sheeted hyperboloid in the Minkowski space $\R^{3,1}$. The central projection again provides a correspondence between the Lagrange problem in $\R^{3}$ restricted to the unit ball and the Lagrange problem in the hyperboloid model of the three-dimensional hyperbolic space. Also, the projective correspondence of reflection walls of billiard systems can be defined and analyzed in the same fashion. The arguments in this section are completely analogous to our previous discussions on the system in $\R^{3}$ and on $\mathcal{S}^{3}$.}

Consider the four-dimensional Minkowski space $\R^{3,1}$ equipped with the pseudo-Riemannian metric
\begin{equation}
	\label{eq: Minkowski_metric}
	dq_1^2 + dq_2^2 +dq_3^2 - dq_4^2{,}
\end{equation}
and the hyperboloid $\mathcal{H}$ in $\R^{3,1}$ given by the equation
\[
q_1^2 + q_2^2 + q_3^2 - q_4^2 = -1.
\]
{With the restriction of the metric \eqref{eq: Minkowski_metric}, the hyperboloid $\mathcal{H}$ is a Riemannian manifold (indeed, also with constant negative curvature)}. We denote by $\|\cdot \|_H$ the Minkowski norm in $\R^{3,1}$ {corresponds to} the metric \eqref{eq: Minkowski_metric}.
We take the lower {sheet} of the hyperboloid 
\[
\mathcal{H}_{S}= \{(q_1,q_2,q_3,q_4)\in \mathcal{H} \mid q_4 <0   \}
\]
and define mechanical systems on it and call them hyperbolic systems. {We denote the upper sheet of $\mathcal{H}$ by $\mathcal{H}_N$.}

We consider the subspace $W_H:= \{(q_1,q_2,q_3,-1)\}$ in $\R^{3,1}$. The central projection from $O= (0,0,0,0) \in \R^{3,1}$ projects $\mathcal{H}_S$ to the {Beltrami-Klein} ball $B:= \{(q_1,q_2,q_3) \in W_H \mid q_1^2 + q_2^2 + q_3^2 < 1  \} \subset W_H$, {which becomes a model of the hyperbolic space should we equip it with the induced metric from the metric on $\mathcal{H}_S$ by the central projection.} We equip $W_H$ with the norm which is induced from the metric \eqref{eq: Minkowski_metric} in $\R^{3,1}$, {which by restriction equips $B$ with another metric which is flat. Note that these two metrics on $B$ are projectively-equivalent, meaning that they have the same unparametrized geodesics, which are simply the intersections of lines in $W_{H}$ with $B$. } Later {on}, we {shall} allow further affine change of the metric {on} $W_H$. {In $B$, by restriction, this results in a different flat metric whose unparametrized geodesics remain the same. This observation was used crucially in the celebrated works of Tabachnikov concerning Birkhoff billiards \cite{Tabachnikov 1999}. 

 Analogously as in the spherical case, we project a force field $F_{W_H}$ on $W_{H}$ to $\mathcal{H}_{S}$ {by making a proper time reparametrization on the push-forward of $F_{W_H}$ 
 by the central projection.  In this way, we obtain} a force field $F_{H}$ on $\mathcal{H}_S$.
  {Precisely,} when we have the force field $F_{W_H}$ on $B \subset W_{H}$,  then the equations of motion of the system in $B$ is 
\[
\ddot{\tilde{q}} = F_{W_H}(\tilde{q})
\]
where $\tilde{q} \in B$ and $\ddot{\tilde{q}}$ is the second time derivative of $\tilde{q}$. When $\tilde{q} \in B$ is projected to $q \in \mathcal{H}_S$:
\begin{equation}
\label{eq: time_parameter_hyp}
q = \|\tilde{q}\|_H^{-1} \tilde{q},
\end{equation}
we have 
\[
\dot{q} = \|\tilde{q}\|_H^{-2}(\dot{\tilde{q}}\| \tilde{q}\|_H - \langle \nabla \| \tilde{q}\|_H, \dot{\tilde{q}} \rangle \tilde{q}).
\]
Here, we take the new time variable $\tau$ satisfying that 
\[
\frac{d}{d \tau } = \| \tilde{q} \|_H^2 \frac{d}{dt}
\]
for the hyperbolic system, and {we denote} the time derivative {by} $':= \frac{d}{d \tau}$. Consequently we have 
\[
{q'} = \dot{\tilde{q}}\| \tilde{q}\|_H - \langle \nabla \| \tilde{q}\|_H, \dot{\tilde{q}} \rangle \tilde{q},
\]
and 
\[
q'' = \| \tilde{q}\|_H^2 \left(\ddot{\tilde{q}}\| \tilde{q}\|_H  -  \left(\langle \nabla \| \tilde{q}\|_H, \ddot{\tilde{q}} \rangle + \left\langle \dfrac{d \nabla \| \tilde{q}\|_H}{dt}, \dot{\tilde{q}}\right\rangle \right) \tilde{q}\right),
\]
where the gradient and the inner product are defined with respect to the Minkowski metric. 
Thus we have 
\[
q'' =  \| \tilde{q}\|_H^2 (F_{W_H}(\tilde{q} - \lambda(\tilde{q}, \dot{\tilde{q}}, \ddot{\tilde{q}})\tilde{q} ),
\]
where $\lambda(\tilde{q}, \dot{\tilde{q}}, \ddot{\tilde{q}}) = \langle \nabla \| \tilde{q}\|_H, \ddot{\tilde{q}} \rangle + \left\langle \dfrac{d \nabla \| \tilde{q}\|_H}{dt}, \dot{\tilde{q}}\right\rangle$. 
By projecting both sides of this equation to $\mathcal{H}_{S}$, we obtain the equations of the motion on $\mathcal{H}_S$ given by 
\[
\nabla_{q'} q'= F_{H}(q),
\]
where $\nabla$ is the Levi-Civita connection of the Riemannian manifold $\mathcal{H}_S$ and with the projected force $F_{H}$ on $\mathcal{H}_S$. 

By switching from the {spherical-Euclidean} projective correspondence to the {hyperboloid-Euclidean} projective correspondence, we obtain analogous results for the hyperbolic systems defined on $\mathcal{H}_S$.

We first define the Lagrange problem in $B \subset W_H$.
\begin{defi}
	\label{def: Lagrange_B}
	The Lagrange problem on $(B, \| \cdot \|_{ia})$ is defined as the restriction to $B$ of the Lagrange problem in $(W_H, \|\cdot \|_{ia})$ with all of its Hooke and  Kepler centers lying in $B$.
\end{defi}

{Next we define {the} Lagrange problem on the hyperboloid $\mathcal{H}_S$. {Note that two center problem in the hyperbolic plane has been defined by Killing {\cite{Killing}}. The extension by adding a Hooke-type center is straightforward.}}

\begin{defi}
	\label{def: Lagrange_hyperbolic}
	{The Lagrange problem on $(\mathcal{H}_S, \| \cdot \|_H)$ with its Kepler centers $Z_1$ and $Z_2$ on $\mathcal{H}_S$ and its Hooke center $Z_0 \in \mathcal{H}_S$ in the middle of two Kepler centers is the system
	\[
	(\mathcal{H}_S, \| \cdot \|_H, m_1 \coth \theta_{Z_1} + m_2 \coth \theta_{Z_2} + f \tanh^2 \theta_{Z_0}),
	\]
	where $\theta_{Z_1}= \angle q OZ_1, \theta_{Z_2}= \angle q OZ_2, \theta_{Z_0}= \angle q OZ_0$ are hyperbolic angles and $m_1, m_2, f \in \R$ are mass parameters. As special cases, we have:
	\begin{itemize}
		\item $f =0$: This system is the two-center problem on the hyperboloid $\mathcal{H}_S$.
		\item $m_1 = f=0$ or $m_2=f=0$: This system is the Kepler problem on the hyperboloid $\mathcal{H}_S$.
		\item $m_1=m_2=0$: This systems is the Hooke problem on the hyperboloid $\mathcal{H}_S$.
	\end{itemize} 
	}
\end{defi}

Here, if a natural mechanical system in $W_H$ has the corresponding projected system defined on $\mathcal{H}_S$ which is derived from a potential, {then} we call the initial system {\emph{h-projective}}.
\begin{defi}
	\label{def: hyp_projective}
	A natural mechanical system in $B$ is called {\emph{h-projective}} if
	the projected system on $\mathcal{H}_{S}$ is also a natural mechanical system, \emph{i.e.,} the projected force field $F_H$ defined on $\mathcal{H}_{S}$ is derived from \emph{a} force function.
\end{defi}

{We now show that the Lagrange problem in $B$ with Hooke center {(which can be ineffective)} at $Z_0=(0,0,0,-1)$ is also h-projective.} Without the loss of generality, we {put} the {Kepler} centers of the Lagrange problem in $W_H$ as
\[
\tilde{Z}_1 = (a, 0,0,-1), \tilde{Z}_2=(-a,0,0,-1),
\]
where $a \in (-1,1)$.
and we define the norm $\|\cdot  \|_{ia}$ of $W_H$ as follows:
\[
\| \tilde{q}\|_{ia} = \sqrt{\frac{q_1^2}{1-a^2}+q_2^2+q_3^2}
\]
for $\tilde{q}= (q_1,q_2,q_3,-1) \in W_{H}$. With this setting, the following lemma can be proven similarly as in the spherical case. 
\begin{lemma}
	\label{lem: hyp_Lagrange_projective}
	The three-dimensional Lagrange problem {in $(B, \|\cdot\|_{ia})$ with the Hooke center at $(0,0,0,-1)$} {is h-projective. }
\end{lemma}
{This lemma implies the integrability of the Lagrange problems defined in $W_H$ and on $\mathcal{H}_S$. Indeed, for the Lagrange problem on $\mathcal{H}_S$, besides its own energy 
\begin{equation}
\label{eq: E_hyp}
E_{hyp}:= \frac{\| q'\|_H^2}{2} - \frac{\hat{m}_1}{\tanh \theta_{Z1}} - \frac{\hat{m}_2}{\tanh \theta_{Z_2}} -f \tanh^2 \theta_{Z_0},
\end{equation}
the {projected} spatial energy in $W_H$ to $\mathcal{H}$ can be written into
\begin{small}
\begin{equation}
\label{eq: E_sp_proj_hyp}
\begin{split}
\hat{E}_{sp} := &\frac{(q_4 q_1'- q_1 q_4')^2}{2  (1-a^2)} + \frac{(q_4 q_2'- q_2 q_4')^2}{2 } + \frac{(q_4 q_3'- q_3 q_4')^2}{2  } - f \left(\frac{q_1^2}{q_4^2(1-a^2)} + \frac{q_2^2}{q_4^2} + \frac{q_3^2}{q_4^2}\right)\\
&- m_1\left(\frac{(-q_1/q_4-a)^2}{1-a^2} + \frac{q_2^2}{q_4^2} + \frac{q_2^3}{q_4^2}  \right)^{-\frac{1}{2}}-m_2\left(\frac{(-q_1/q_4+a)^2}{1-a^2} + \frac{q_2^2}{q_4^2} + \frac{q_2^3}{q_4^2}  \right)^{-\frac{1}{2}} 
\end{split}
\end{equation}
\end{small}
is the first integral. {The other first integral is the component of the angular momentum}
\begin{equation}
\label{eq: L_23_proj_hyp}
\hat{L}_{{yz}} := q_3 q_2' - q_2 q_3'.
\end{equation}
}
We now add the reflection walls and consider the corresponding hyperbolic billiard systems on $\mathcal{H}_S$. {Analogously as} in the spherical case, we {need to} require {the} reflection walls to be {confocal} to establish integrability. 

{We set the Kepler centers $Z_1$ and $Z_2$ on $\mathcal{H}_S$ as the projections of $\tilde{Z}_1, \tilde{Z}_2$:
\[
Z_1= \left( \frac{a}{\sqrt{1-a^2}}, 0,0, -\frac{1}{\sqrt{1-a^2}} \right), Z_2= \left( -\frac{a}{1-a^2}, 0,0, -\frac{1}{\sqrt{1-a^2}} \right),
\]
}
{Foci for quadrics on the hyperboloid model $\mathcal{H}_S$ are defined in the same way as the Euclidean and spherical cases but with the hyperbolic distance.}
\begin{defi}
	\label{def: symmetric_quadrics_H3}
	{A quadric with two foci in the hyperboloid model $(\mathcal{H}_S, \| \cdot \|_H)$ is called confocal if it is focused at two Kepler centers $Z_1,Z_2$.}
\end{defi}

{Such a quadric is centered at the Hooke center $Z_0 = (0,0,0,-1)$, and is, after normalization, given by 
\begin{displaymath}
	\left\{
	\begin{array}{l}
		q_{1}^{2}+q_{2}^{2}+q_{3}^{2}{-}q_{4}^{2}=-1, \quad q_4 <0 \\
		\frac{q_1^2 }{A^2} + \frac{q_2^2}{B^2} + \frac{q_3^2}{B^2} - q_4^2 = 0, \quad {A,B>0} 
	\end{array}
	\right.
\end{displaymath}
for a spheroid, and by
\begin{displaymath}
	\left\{
	\begin{array}{l}
		q_{1}^{2}+q_{2}^{2}+q_{3}^{2}{-}q_{4}^{2}=-1, \quad q_4 <0 \\
		\frac{q_1^2 }{A^2} - \frac{q_2^2}{B^2} - \frac{q_3^2}{B^2} - q_4^2 = 0, \quad {A,B>0}
	\end{array}
	\right.
\end{displaymath}
for a circular hyperboloid of two sheets. }

{{With the same argument} as in the spherical case, the angular momentum $\hat{L}_{yz}$ is conserved under reflections at confocal quadrics on $\mathcal{H}_{S}$}.
{Indeed, such} confocal quadric reflection walls in $W_H$ and on $\mathcal{H}_S$ are in projective correspondence, \emph{i.e.,} the central projection and reflections at such reflection walls commute. 

\begin{lemma}
	\label{lem: hyp_reflection_law}
	Let $\mathcal{\tilde{B}}$ be {the restriction} of {a}  confocal quadric reflection wall {in $W_{H}$ to} {$(B, \| \cdot \|_{ia})$}
		with foci at $\tilde{F}_1=(a,0,0)$ and  $\tilde{F}_2=(-a,0,0)$. {Let} $\mathcal{B}$ be {the 
		quadric} on $\mathcal{H}_{S}$ given as the projection of $\mathcal{\tilde{B}}$. 
		Then $\mathcal{B}$ is a 
		confocal quadric {having} its foci at the projections of $\tilde{F}_1$ and $\tilde{F}_2$ on $\mathcal{H}_{S}$. {Moreover, $\mathcal{\tilde{B}}$ and $\mathcal{B}$ are in projective correspondence.} 
\end{lemma}
{This lemma follows from a similar argument in the proof of Lemma \ref{lem: reflection_law} using the space-hyperboloid correspondence.}

{{Consequently}, in addition to the angular momentum $\hat{L}_{yz}$, the first integral $\hat{E}_{sp}$ is also conserved under reflections at confocal quadrics on $\mathcal{H}_S$.}

By combining Lemma \ref{lem: hyp_Lagrange_projective} and Lemma \ref{lem: hyp_reflection_law}, we obtain the following result. 

\begin{theorem}
	{The 
	 mechanical billiard systems defined on the {hyperboloid model} $\mathcal{H}_S$ 
	 with the {hyperbolic} Lagrange problem 
	 and with any {finite} combination of confocal 
	 quadrics
	 with foci at the two Kepler centers as reflection walls, are integrable.}
	
	
	As subcases, {the mechanical billiard problems defined on $\mathcal{H}_{S}$}
		\begin{itemize}
			\item with the hyperbolic two-center problem and with any {finite} combination of confocal 
			{quadrics} with foci at the two centers as reflection walls; 
			\item with the hyperbolic Kepler problem and with any {finite} combination of confocal 
			{quadrics} with one of the foci at the Kepler center as reflection walls;
			\item  with the 
			hyperbolic 
			Hooke problem and with any {finite} combination of confocal 
			 quadrics centered at the Hooke center as reflection walls, 
		\end{itemize}	
		are integrable. {The first integrals are $E_{hyp}$ \eqref{eq: E_hyp}, $\hat{E}_{sp}$ \eqref{eq: E_sp_proj_hyp}, and $\hat{L}_{yz}$ \eqref{eq: L_23_proj_hyp}.}
\end{theorem}

%

\section{{The Higher Dimensional Cases}}
\label{sec: higher_dimensional}
We now discuss the Lagrange problem defined in higher dimensional space forms and the integrability of the corresponding billiard systems. 

In the {$(n+1)$}-dimensional Euclidiean space $\R^{n+1}$, we consider the $n$-dimensional subspace $W^n:=\{ (q_1, q_2, \cdots q_{n}, -1 )  \} \subset \R^{n+1}$ and the $n$-dimensional unit sphere 
\[
\mathcal{S}^{n}: = \{(q_1, q_2, \cdots, q_{n+1})\in \R^{n+1}  \mid q_1^2 + q_2^2 + \cdots + q_{n+1}^2 =1  \}.
\] 
The central projection from the origin $O$ in $\R^{n+1}$ projects the points of south-hemisphere $\mathcal{S}^n_{SH}:=\{(q_1,q_2, \cdots, q_{n+1}) \in \mathcal{S}^n \mid q_{n+1}<0  \}$ to the points in $W^n$. We equip $\mathcal{S}^n$ and $\mathcal{S}^n_{SH}$ with their induced round metric from $\R^{n+1}$. We denote the Euclidean inner product and the Euclidean norm of $\R^{n+1}$ by $\langle \cdot , \cdot  \rangle$ and $\| \cdot \|$, respectively. 

Analogously as in the three-dimensional case, a force field $F_W$ defined on ${W^n}$ is projected to a force field $F_S$ on $\mathcal{S}^n_{SH}$ by the push-forward of the central projection with the factor of time change which is uniquely determined by the projection. 

Consider the motion of the system in $W^n$ given by 
\begin{equation}
\label{eq: newton_eq_Wn}
\ddot{\tilde{q}} = F_W(\tilde{q}),
\end{equation}
where $\tilde{q} \in W^n$. Let $q \in \mathcal{S}^n_{SH}$ and $\tilde{q} \in W^n_{H}$ are projectively corresponded \emph{i.e.}
\[
q = \| \tilde{q}\|^{-1} \tilde{q},
\]
then we obtain
\[
\dot{q} = \|\tilde{q}\|^{-2} (\dot{\tilde{q}}\|\tilde{q} \| - \langle \nabla \| \tilde{q} \|, \dot{\tilde{q}}  \rangle \tilde{q} )
\]
where $\dot{} := \frac{d}{dt}$ is the {time} derivative. {We take} a new time variable $\tau$ for the system on $\mathcal{S}^n_{SH}$ such that 
\begin{equation}
	\label{eq: time_parameter_change_ndim}
	{\dfrac{d}{d \tau}= \| \tilde{q} \|^{2}\dfrac{d}{d t}}
\end{equation}
and write as $' :=\dfrac{d}{d\tau}$. We then have 
\[
q' = \dot{\tilde{q}}\|\tilde{q} \| - \langle \nabla \| \tilde{q} \|, \dot{\tilde{q}}  \rangle \tilde{q}   
\]
and 
\[
q'' = \tilde{q}\|^2 (\ddot{\tilde{q}} \| \tilde{q}\| - (\langle \nabla \| \tilde{q} \|,  \ddot{\tilde{q}}  \rangle  + \langle d (\nabla \| \tilde{q}\|)/dt, \tilde{q} \rangle  )  \tilde{q}).
\]
{Now by plugging} \eqref{eq: newton_eq_Wn} into above equation, we obtain 
\[
q'' = \| \tilde{q} \|^2 (F_W(\tilde{q}) \| \tilde{q}\| - \lambda(\tilde{q}, \dot{\tilde{q}}, \ddot{\tilde{q}}) \tilde{q}  ),
\]
where $\lambda(\tilde{q}, \dot{\tilde{q}}, \ddot{\tilde{q}}) =\langle \nabla \| \tilde{q} \|,  \ddot{\tilde{q}}  \rangle  + \langle d (\nabla \| \tilde{q}\|)/dt, \tilde{q} \rangle $.
By projecting both sides of this equations to the tangent space {$T_q \mathcal{S}^n_{SH}$} we get the equations of motion on $\mathcal{S}^n_{SH}$ given in the form 
\[
\nabla_{q'} q' = F_{S}(q).
\]

The Lagrange problem in $\R^n$, on $\mathcal{S}^n_{SH}$, and on $\mathcal{S}^n$ are defined analogously as Definition \ref{def: Lagrange_R3}, \ref{def: Lagrange_def_hemisph3}, and \ref{def: Lagrange_ded_sph3}, respectively.

By the method of projective dynamics, we can see that the Lagrange problem in $W^n$ with the Hooke center at $Z_0 = (0,0, \cdots,0, -1) \in \R^{n+1}$ is projective {in the analogous sense of Definition \ref{def: projective}:}

\begin{defi}
	A natural mechanical system in $W^n$ is called projective if the projected system on $\mathcal{S}^n_{SH}$ is also a natural mechanical systems. 
\end{defi}

 {Without loss of generality, we assume} that the Kepler centers {{in} the spatial Lagrange problem in $W^n$, and {in} the spherical Lagrange problem on $\mathcal{S}^n$ are placed on the $q_{1}$-axis {respectively} with $q_{1} = a$ and $q_1= -a$,} for some $a \in \R$. {We define} the norm $\| \cdot \|_a$ in $W^n$ as 
\[
\| \tilde{q}\|_a =\left(\frac{q_1^2}{1+a^2}+ \sum_{i=2}^n q_i^2 \right)^{\frac{1}{2}},
\]
{for $\tilde{q}= (q_1, q_2, \cdots, q_n) \in W^n$.}
With this setting, the following lemma can be proven similarly as in the three-dimensional case. 

\begin{lemma}
	\label{lem: Lagrange_proj_ndim}
	The Lagrange problem in {$(W^n, \|\cdot \|_{a})$, with the Hooke center at $\tilde{Z}_0 =(0,0,\cdots, 0,-1)$} is projective.
\end{lemma}

{This lemma implies that both the Lagrange problems in $\R^n$ and on $\mathcal{S}^n_{HS}$ are integrable. Indeed, the system in ${W^n}$ has {its energy
\begin{equation}
\label{eq: E_sp_ndim}
E_{sp}:=\frac{\|\dot{q}\|^2_a}{2} -f \| \tilde{q}\|_a^2- m_1 \| \tilde{q} - \tilde{Z}_1 \|_a^{-1} -  m_2 \| \tilde{q} - \tilde{Z}_2 \|_a^{-1}
\end{equation}
}
and {an} additional first integral {obtained from} the spherical energy {whose representation in the chart of $W^n$ is given by
\small
\begin{equation}
\label{eq: E_sph_proj_ndim}
\begin{split}
\tilde{E}_{sph}:= &\frac{1}{2} \left(\sum_{i=1}^n \dot{q_i}^2 + \sum_{1 \leq i<j\leq n} (q_i \dot{q_j} - q_j \dot{q_i})^2 \right) - f \sum_{i=1}^n q_i^2 \\& - \frac{\hat{m}_1 (1+a q_1)}{(q_1-a)^2+(1+a^2)(\sum_{i=2}^n q_i^2) }-\frac{\hat{m}_2 (1-a q_1)}{(q_1+a)^2+(1+a^2)(\sum_{i=2}^n q_i^2)},
\end{split}
\end{equation}
{where $\hat{m}_1 = m_1 \sqrt{1+a^2}, \hat{m}_2 = m_2 \sqrt{1+a^2}$.}
\normalsize
} {Additionally, the angular momentum bi-vector can be represented by an antisymmetric $n \times n$-matrix with entries }
\begin{equation}
\label{eq:Lij_ndim}
L_{{ij}}:= q_i \dot{q_j} - \dot{q_i} q_j, \quad {1} \leq i, j \leq n.
\end{equation}
{{Due to the position of the centers, the components 
$$L_{{ij}}, \qquad 2 \leq i,j \leq n $$ 
of the angular momentum} are preserved in the Lagrange system}. {They are not pairwise in involution, but from them, the following $n-2$ independent first integrals pairwise in involution are obtained} 
\begin{equation}
\label{eq: Ck_ndim}
C_k:= \sum_{2 \leq i < j \leq k} L_{{ij}}^2, \qquad 3 \leq k \leq n.
\end{equation}

\begin{lemma}
	{The $n$ functions $ E_{sp}, \tilde{E}_{sph}, C_3, C_4, \cdots, C_n$ are functionally independent and mutually in involution. Consequently, the Lagrange problem in $W^{n}$ is integrable.}
\end{lemma}
\begin{proof}
{{It is not hard to see that they are functionally independent.} Indeed, if we consider the Jacobian matrix
\begin{equation*}
	\begin{pmatrix}
		\frac{\partial C_{3}}{\partial \partial q_2} & \frac{\partial C_{3}}{\partial \partial q_3} & \cdots & \frac{ \partial C_3}{\partial q_n}& \frac{\partial C_{3}}{\partial \partial \dot{q_2}} & \frac{\partial C_{3}}{\partial \partial \dot{q_3}} & \cdots & \frac{ \partial C_3}{\partial \dot{q_n}} \\
		\frac{\partial C_{4}}{\partial \partial q_2} & \frac{\partial C_{4}}{\partial \partial q_3} & \cdots & \frac{ \partial C_4}{\partial q_n} & \frac{\partial C_{4}}{\partial \partial \dot{q_2}} & \frac{\partial C_{4}}{\partial \partial \dot{q_3}} & \cdots & \frac{ \partial C_4}{\partial \dot{q_n}}\\
		\vdots & \vdots & \ddots & \vdots & \vdots & \vdots & \ddots & \vdots \\
		\frac{\partial C_{n}}{\partial \partial q_2} & \frac{\partial C_{n}}{\partial \partial q_3} & \cdots & \frac{ \partial C_n}{\partial q_n} & \frac{\partial C_{n}}{\partial \partial \dot{q_2}} & \frac{\partial C_{n}}{\partial \partial \dot{q_3}} & \cdots & \frac{ \partial C_n}{\partial \dot{q_n}}\\
	\end{pmatrix}
\end{equation*}	
and its $(n-2)\times (n-1)$-submatrix given by 
\begin{equation}
	\begin{pmatrix}
		\frac{\partial C_{3}}{\partial \partial q_2} & \frac{\partial C_{3}}{\partial \partial q_3} & \cdots & \frac{ \partial C_3}{\partial q_n} \\
		\frac{\partial C_{4}}{\partial \partial q_2} & \frac{\partial C_{4}}{\partial \partial q_3} & \cdots & \frac{ \partial C_4}{\partial q_n} \\
		\vdots & \vdots & \ddots & \vdots \\
		\frac{\partial C_{n}}{\partial \partial q_2} & \frac{\partial C_{n}}{\partial \partial q_3} & \cdots & \frac{ \partial C_n}{\partial q_n} \\
	\end{pmatrix}
\end{equation}
whose $(i,j)$-entries for $i+1<j$ are zeros and other entries are non-zero. Hence  this submatrix and also the original Jacobian matrix has {rank $n-2$}.
}

\marginpar{
}
{
{To see that they are also in involution} , we use the following relation on the Poisson bracket of angular momentum {components}:
\[
\{L_{k_1 k_2 }, L_{ \ell_1 \ell_2}    \} = \delta_{k_1,\ell_1}L_{k_2 \ell_2}+ \delta_{k_2,\ell_2}L_{k_1 \ell_1} - \delta_{k_1,\ell_2}L_{k_2 \ell_1} - \delta_{k_2,\ell_1}L_{k_1 \ell_2}. 
\]
{From this relation, when $2< k < \ell_1 <\ell_2 \leq n$, we have 
\begin{equation*}
\begin{split}
\{C_k,L_{\ell_1 \ell_2}\}&= \sum_{2<k_1<k_2\leq k} \{  L_{k_1 k_2 }^2, L_{ \ell_1 \ell_2}   \}  \\
&= 2\sum_{2<k_1<k_2\leq h} \{  L_{k_1 k_2 }, L_{ \ell_1 \ell_2}   \} \\
& = 0 
\end{split}
\end{equation*}
    since all indices $k_1, k_2, \ell_1, \ell_2$ are different.  

Additionally, when 
$$2 \leq \ell_1 <\ell_2 \leq k \leq n, $$
we have that
\small
\begin{align*}
\{C_k,L_{\ell_1 \ell_2}\} &=\sum_{2<k_1<k_2\leq k} \{  L_{k_1 k_2 }^2, L_{ \ell_1 \ell_2}   \} \\
&=2\sum_{2<k_1<k_2\leq k} \{  L_{k_1 k_2 }, L_{ \ell_1 \ell_2}   \} \\
&= 2\left(  \sum_{\ell_1 < k_2 \leq k }L_{k_2 \ell_2}L_{\ell_1 k_2}+ \sum_{2< k_1 <\ell_2}L_{k_1 \ell_1}L_{k_1 \ell_2}\right)\\
 &\phantom{=} -2\left(\sum_{2< k_1 <\ell_1}L_{k_1 \ell_2} L_{k_1 \ell_1} +\sum_{\ell_2< k_2 \leq k}L_{k_2 \ell_1} L_{\ell_2 k_2 }\right)\\
&=2 \left(\sum_{\ell_1< k_2 \leq \ell_2} L_{k_2 \ell_2 } L_{\ell_1 k_2} + \sum_{\ell_1 \leq k_1 <\ell_2} L_{k_1 \ell_1} L_{k_1 \ell_2} \right)\\
& = 0.
\end{align*}
\normalsize
Combining these cases, we obtain 
\[
\{C_k, L_{\ell_1\ell_2}  \}=0
\]
when $ \ell_1<\ell_2 \leq k, k \leq \ell_1<\ell_2$.
}
Using this, we get
\[
\{C_k, C_{\ell}  \} = 0 
\]
for any $3 \leq k, \ell \leq n$. 

Also, they are independent of the spatial energy $E_{sp}$ and the {projected} spherical energy $\tilde{E}_{sph}$, and are Poisson commutative with both energies. Indeed, since $C_{k}$ for $3 \leq k \leq n$ are first integrals of the spatial system, we have
\[
\{E_{sp}, C_k  \} =0, \quad 3 \leq k \leq n.
\]
{The other Poisson-commutativity }
\[
\{\tilde{E}_{sph}, C_k  \} =0, \quad 3 \leq k \leq n
\]
follows from the same argument as in the three-dimensional case {explained in Section \ref{sec: main_resutls}}, that the spherical energy is invariant under the rotation with respect to the {$q_1$-axis,} \emph{i.e.}
\[
\{ \tilde{E}_{sph}, L_{ij}  \} =0, \quad \quad 2 \leq i,j \leq n.
\]
}
	
\end{proof}

{The spherical Lagrange system on $\mathcal{S}^n_{SH}$ {has the} energy 
\begin{equation}
\label{eq: E_sph_ndim}
E_{sph}:= \frac{1}{2} \sum_{i=1}^{n+1} q_i'^2 - \frac{\hat{m}_1}{\tan \theta_{Z_1}}  - \frac{\hat{m}_2}{\tan \theta_{Z_2}} - f \tan^2 \theta_{Z_0},
\end{equation}
{where $q_i'$ is the time derivative with respect to the time parameter $\tau$ defined as \eqref{eq: time_parameter_change_ndim}.}
{Besides this,} the projection of $E_{sp}$ expressed in the external coordinates for $\mathcal{S}^n$ as:
\footnotesize
\begin{equation}
\label{eq: E_sp_proj_ndim}
\begin{split}
	\hat{E}_{sp}:= &\frac{1}{2} \left(\frac{(q_{n+1}q_1' - q_1 q_{n+1}')^2}{1+a^2} +\sum_{i =2}^n ( q_{n+1} q_i' - q_i q_{n+1}' )^2  \right)-f\left(\frac{q_1^2}{q_4^2(1+a^2)} + \sum_{i =2}^{n} \frac{q_i^2}{q_{n+1}^2}  \right)\\
	& -m_1\left(\frac{(-q_1/q_{n+1} -a)^2}{1+a^2} + \sum_{i=2}^n \frac{q_i^2}{q_{n+1}^2}    \right)^{-1/2} - m_2 \left(\frac{(-q_1/q_{n+1} +a)^2}{1+a^2} + \sum_{i=2}^n \frac{q_i^2}{q_{n+1}^2}    \right)^{-1/2}, 
	\end{split}
\end{equation}
\normalsize
{where $m_1 = \frac{\hat{m}_1}{\sqrt{1+a^2}}, m_2 = \frac{\hat{m}_2}{\sqrt{1+a^2}}$,} is preserved.}

{{The {projections} of the angular momentum components}(which are just the components of the angular momentum in $\R^{n+1}$) are given by
	\[
	\hat{L}_{ij} = q_i q_j' - q_i' q_j, \quad 1 \leq i, j \leq n.
	\]
	From them we obtain the $n-2$ independent first integrals for the spherical Lagrange system on $\mathcal{S}^n_{SH}$:
	\begin{equation}
	\label{eq: Ck_proj_ndim}
	\hat{C}_k:= \sum_{2 \leq i < j \leq k} \hat{L}_{{i}j}^2, \qquad 3 \leq k \leq n.
	\end{equation}
	They are functionally independent and in involution as well as {in} the {previous case.} 
}

{We now consider billiard systems on $n$-dimensional space forms with the Lagrange problems defined on them. 
	As in the three-dimensional case, we require that these quadrics are {confocal} in the following sense:}
	
First, for $\R^{n}$ we equip it with an appropriate norm. We fix the Kepler centers in $\R^n$ at $\tilde{Z}_1= (a, 0, \cdots, 0), \tilde{Z}_2= (-a, 0, \cdots, 0 )$.

{Foci for quadrics in $\R^n$ are defined in the same way as in $\R^3$.}
\begin{defi}
	{A quadric with two foci in $\R^n$ is called confocal if it is focused at  $\tilde{Z}_1$ and $\tilde{Z_2}$.}
\end{defi}

{Such a quadric is centered at the Hooke center $\tilde{Z}_0 = (0, \cdots, 0)$, and is, after normalization, given by 
\[
\frac{q_1^2}{A^2} + \sum_{i = 1}^{n} \frac{q_i^2}{B^2} = 1, \quad {A, B >0}
\]
for a {(higher-dimensional)} spheroid, and by
\[
\frac{q_1^2}{A^2} - \sum_{i = 1}^{n} \frac{q_i^2}{B^2} = 1, \quad {A, B >0}
\]
for a {(higher-dimensional)} circular hyperboloid of two sheets.}

{For the spherical Lagrange problem, we consider spherical confocal quadrics in $\mathcal{S}^n$ as reflection {walls} which are defined as follows.}
{We set the Kepler centers $Z_1,Z_2$ on $\mathcal{S}^n_{SH}$ as the projection of $\tilde{Z}_1, \tilde{Z}_2$:
\[
Z_1 = \left(\frac{a}{\sqrt{1+a^2}}, 0, \cdots, 0, -\frac{1}{\sqrt{1+a^2}}\right), Z_2 = \left(-\frac{a}{\sqrt{1+a^2}}, 0, \cdots, 0, -\frac{1}{\sqrt{1+a^2}}\right)
\]
We equip $\mathcal{S}^n$ with the round metric induced from the standard Euclidean metric of $\R^{n+1}$.
}

{Foci for spherical quadrics on $\mathcal{S}^n$ are defined in the same way as on $\mathcal{S}^3$.}
\begin{defi}
	{A spherical quadric with two foci on $\mathcal{S}^n_{SH}$ or on $\mathcal{S}$ is called confocal if it is focused at $Z_1,Z_2$.}
\end{defi}

{A confocal quadric on $\mathcal{S}^n_{SH}$ has its center at the Hooke center $Z_0 = (0,\cdots, 0, -1)$ and is, after normalization, given by 	\begin{equation}
	\left\{
	\begin{array}{l}
		\sum_{i=1}^{n+1} q_i^2 =1, \quad q_{n+1}<0 \\
		\frac{q_1^2}{A^2}+\sum_{i=2}^{n} \frac{q_i^2}{B^2} - q_{n+1}^2 = 0, \quad {A,B>0}
	\end{array}
	\right.
\end{equation}
for a {(higher-dimensional)} spheroid, and by
\begin{equation}
	\left\{
	\begin{array}{l}
		\sum_{i=1}^{n+1} q_i^2 =1,\quad q_{n+1}<0 \\
		\frac{q_1^2}{A^2}-\sum_{i=2}^{n} \frac{q_i^2}{B^2} - q_{n+1}^2 = 0, \quad {A,B>0}
	\end{array}
	\right.
\end{equation}
for a {(higher-dimensional)}circular hyperboloid of two sheets.}

{These reflection walls are highly symmetric, and the first integrals {$\{C_{k}\}_{k=3}^n$} obtained from angular momentum for the Lagrange problem in $\R^n$ are conserved under reflections at them.} 
This immediately follows from the lemma below. 

\begin{lemma}
	\label{lem: angular_momentum_pres_ndim}
	{The angular momentum $L_{ij}$ for $2 \leq i,j \leq n$ are conserved under reflections at a confocal quadric reflection wall $\tilde{\mathcal{B}}$ in {$(W^n, \| \cdot \|_a)$}.
	}
\end{lemma}
\begin{proof}
	{Let $\mathcal{B}$ be {a spheroid} given by the normalized equation
	\[
	\frac{q_1^2}{A^2} + \sum_{k =2}^{n} \frac{q_k^2}{B^2} = 1
	\]
	Fix the integers $i,j$ such that $2 \leq i,j \leq n$ and consider the angular momentum $L_{i j}$ before and after the reflection at $\tilde{q} \in \tilde{\mathcal{B}}$. Without loss of generality, we can assume that the reflection point $\tilde{q} = (q_1, q_2, \cdots, q_n) \in \tilde{\mathcal{B}}$ has zero $q_j$-component due to the symmetry of $\tilde{\mathcal{B}}$.
	Let $\dot{\tilde{q}} = (\dot{q_1}, \dot{q_2}, \cdots, \dot{q_n})$ be the velocity vector before the reflection at $\tilde{q}$. From the similar computation as in the proof of Theorem \ref{thm: spatial_Lagrange}, we obtain that the $q_j$-components of the velocity vector $v$ after the reflection is given by 
	\[
	v_{j} = \dot{q_j}.
	\] 
	Thus, the angular momentum $L_{i j}$ after the reflection is obtained as
	\[
	v_{j}q_i = \dot{q_j} q_i
	\]
	which is precisely $L_{i j}$ before the reflection.}
	
	{The case of {a circular hyperboloid of two sheets} as wall can be treated similarly.}
\end{proof}

{{Similarly, one can directly check} that the first integrals $\hat{C}_k, 3 \leq k \leq n$ for the spherical Lagrange problem on $\mathcal{S}^n_{SH}$ are conserved under reflections at  a confocal quadrics on $\mathcal{S}^n_{SH}$. Alternatively this follows form the lemma below.} 

We can see that such confocal quadric reflection walls in {$(W^n,  \| \cdot\|_{a})$} and on $\mathcal{S}^n_{SH}$ are in projective correspondence as well, {by the very same argument} as in the three-dimensional case{:}

\begin{lemma}
	\label{lem: reflection_ndim}
	Let $\tilde{\mathcal{B}}$ be a centered confocal quadric reflection wall in $(W^n, \| \cdot\|_{a})$ with foci at $F_1= (a, 0, \cdots, 0)$ and $F_2 = (-a,0, \cdots, 0)$. Let $\mathcal{B}$ be the quadric on {$\mathcal{S}^n_{SH}$} given as the projection of $\tilde{\mathcal{B}}$. Then $\mathcal{B}$ is a centered confocal quadric having its foci at the projection of $F_1$ and $F_2$ on $\mathcal{S}^n_{SH}$. Moreover, $\tilde{\mathcal{B}}$ and $\mathcal{B}$ are in projective correspondence. 
\end{lemma}

{From this lemma, we {conclude} that $\tilde{E}_{sph}$ is preserved under reflection at a confocal quadric in $W^n$.  {Similarly}, $\hat{E}_{pl}$ is preserved under reflections at a spherical confocal quadric on $\mathcal{S}_{SH}^n$. } {Additionally, the preservation of $\{\hat{C}_k\}_{k=3}^n$ also follows from Lemma \ref{lem: angular_momentum_pres_ndim} and Lemma \ref{lem: reflection_ndim}.}

By combining Lemma \ref{lem: Lagrange_proj_ndim} and Lemma \ref{lem: reflection_ndim}, {we obtain} 
\begin{theorem}
	{The mechanical billiard systems defined in the $n$-dimensional Eulidean space with the {Lagrange problem} and with any {finite} combination of confocal quadrics with {foci} at the two Kepler centers as reflection walls, are integrable.}
	
	{As subcases, we obtain the following results{:} The mechanical billiard problems defined in the $n$-dimensional Euclidean space
	\begin{itemize}
		\item with the $n$-dimensional two-center problem with any finite combination of confocal quadrics with foci at the two centers as reflection walls;
		\item with the $n$-dimensional Kepler problem and with any finite combination of confocal quadrics with one of the foci at the Kepler center as reflection walls;
		\item with the $n$-dimensional Hooke problem and with any finite combination of confocal quadrics centered at the Hooke center as reflection walls,
	\end{itemize}
	}
	{are all integrable.} {The first integrals are $E_{sp}$ \eqref{eq: E_sp_ndim}, $\tilde{E}_{sph}$ \eqref{eq: E_sph_proj_ndim}, and {the functions $C_k$} \eqref{eq: Ck_ndim}} for {$3 \leq k \leq n $}. 
\end{theorem}


{{For the spherical case}, we can extend the integrability to the corresponding billiard system defined on the whole $n$-dimensional sphere since we can analytically extend the $n$ first integrals defined on $\mathcal{S}^n_{SH}$ for the hemispherical system to the whole sphere $\mathcal{S}^n$.}
}}
{Indeed, by analyticity, the kinetic part of $\hat{E}_{sp}$ extends to the {whole} sphere $\mathcal{S}^n$ and the potential part of $\hat{E}_{sp}$ extends to $\mathcal{S}^n$, outside of the singularities at the Kepler centers and their {antipodal points} and at the equator $\{ q_{n+1} =0 \}\cap \mathcal{S}^n$ {when the Hooke center indeed exists}. The other first integrals $\hat{C}_k$ extends to $\mathcal{S}^n$ due to the {constructions} of the angular momentum components $\hat{L}_{i j}$.} {See also the extension argument for the three-dimensional case in the proof of Cororally \ref{cor: integrability_dim_3}.} 

\begin{theorem}
	{The spherical mechanical billiard systems defined on $\mathcal{S}^{n}$ with the Lagrange problem and with any {finite} combination of confocal symmetric centered quadrics, with foci at the two Kepler centers as reflection walls{,} are integrable.}
	
	{As subcases, we obtain the following results{:} The mechanical billiard problems defined on the $n$-dimensional sphere $\mathcal{S}^n$}
	\begin{itemize}
		\item {with the $n$-dimensional spherical two-center problem with any finite combination of confocal quadrics on $\mathcal{S}^n$  with foci at the two centers as reflection walls;}
		\item {with the $n$-dimensional spherical Kepler problem and with any finite combination of confocal quadrics on $\mathcal{S}^n$ with one of the foci at the Kepler center as reflection walls;}
		\item {with the $n$-dimensional spherical Hooke problem and with any finite combination of confocal quadrics on $\mathcal{S}^n$ centered at the Hooke center as reflection walls,}
	\end{itemize}
	are all integrable. {The first integrals are $E_{sph}$ \eqref{eq: E_sph_ndim}, $\hat{E}_{sp}$ \eqref{eq: E_sp_proj_ndim}, and $\hat{C}_k$ \eqref{eq: Ck_proj_ndim} for $3 \leq k \leq n$}.
\end{theorem}


{Analogously, we obtain the integrability for the Lagrange billiard systems {with confocal quadrics} in the $n$-dimensional hyperbolic space. For this purpose, we {consider} the $n+1$-dimensional Minkowski space $\R^{n,1}$ equipped with the pseudo-Riemannian metric given by
\begin{equation}
\label{eq: Minkowski_metric_ndim}
\sum_{i=1}^n dq_i^2 - dq_{n+1}^2, 
\end{equation}
and the hyperboloid {model} $\mathcal{H}^n$ in $\R^{n,1}$ given by the equation 
\[
\sum_{i=1}^n q_i^2 - q_{n+1}^2 = -1. 
\]
The hyperboloid {model} $\mathcal{H}^n$ equipped with the hyperbolic metric \eqref{eq: Minkowski_metric_ndim} is a Riemannian manifold. We take the lower sheet of the hyperboloid 
\[
\{(q_1, q_2, \cdots, q_{n+1}) \in \mathcal{H}^n \mid q_{n+1}<0   \}
\]
{and denote it by $\mathcal{H}^n_S$.}

{We now consider the $n$-dimensional subspace 
$$W^n_H:= \{(q_1, q_2. \cdots , q_n, -1) \} \subset \R^{n,1}.$$
The central projection projects $\mathcal{H}^n_S$ to the $n$-dimensional Beltrami-Klein ball 
$$B^n:=\{(q_1, q_2, \cdots, q_n) \in W^n_H \mid  \sum_{i=1}^n q_i^2 <1  \} \subset W^n_H.$$
We equip $W^n_H$ {with the norm  induced} from the metric of $\R^{n,1}$, {which subsequently induces a norm on $B^n$.} 

As in the spherical case, we project a force field on $B^n$ to $\mathcal{H}^n_S$ with the new time variable $\tau$ 
on the push-forward $F_{W^n_H}$ by the central projection.}

The Lagrange problem in $\mathcal{H}^n_{S}$ is defined analogously as Definition \ref{def: Lagrange_hyperbolic}. {Its energy is}
\begin{equation}
\label{eq: E_hyp_ndim}
E_{hyp}:= \frac{1}{2}\sum_{i=1}^{n+1}q_i'^2 - \frac{\hat{m}_1}{\tanh \theta_{Z_1}} - \frac{\hat{m}_2}{\tanh \theta_{Z_2}} - f \tanh \theta_{Z_0},
\end{equation}
where  $':=d /d\tau$ is the time derivative with respect to $\tau$ {which is defined analogously as \eqref{eq: time_parameter_hyp} in the three-dimensional case}.

{With a similar argument as in Section \ref{sec: hyperboloid}, we get that} the Lagrange problem in $\mathcal{W}^n_S$ is \emph{h-projective} {in the analogous sense of Definition \ref{def: hyp_projective}.} 
}

\begin{defi}
	A natural mechanical system in $B^n$ is called h-projective if the projected system in $\mathcal{H}^n_S$ is also a natural mechanical system. 
\end{defi}

Without loss of generality, {we assume} that the Lagrange problem in $B^n$ has its Kepler centers at $Z_1 =(a, 0, \cdots, 0,-1), Z_2 = (-a, 0, \cdots, 0,-1) \in B^n$ {with $0< a<1$}. {We define the norm} of $W^n_H$, and thus on $B^n$, as
\[
\| \tilde{q} \|^2_{ia} = \frac{q_1^2}{1-a^2} + \sum_{i=2}^n q_i^2
\]
for $\tilde{q} = (q_1, \cdots, q_n, -1) \in W^n_{H}$.

The Lagrange problem in $B^n$ is defined analogously as Definition \ref{def: Lagrange_B}.

\begin{lemma}
	\label{lem: Lagrange_proj_ndim_hyp}
	{The Lagrange problem in $(B^n, \|\cdot \|_{ia})$ is h-projective.}
\end{lemma}

{{This implies the integrability of the Lagrange problems in $B^n$ and in $\mathcal{H}^n_{S}$}. 
The Lagrange problem in $\mathcal{H}_S^n${,} which is the projection of the Lagrange problem in $W_n^H${,} has a first integral {in addition to its own energy $E_{hyp}$}{:} the projection of {the energy of} the spatial Lagrange problem in $B^n$: 
\footnotesize
\begin{equation}
\label{eq: E_sp_proj_ndim_hyp}
\begin{split}
\hat{E}_{sp}:= &\frac{1}{2} \left(\frac{(q_{n+1}q_1' - q_1 q_{n+1}')^2}{1-a^2} +\sum_{i =2}^n ( q_{n+1} q_i' - q_i q_{n+1}' )^2  \right)-f\left(\frac{q_1^2}{q_4^2(1-a^2)} + \sum_{i =2}^{n} \frac{q_i^2}{q_{n+1}^2}  \right)\\
& -m_1\left(\frac{(-q_1/q_{n+1} -a)^2}{1-a^2} + \sum_{i=2}^n \frac{q_i^2}{q_{n+1}^2}    \right)^{-1/2} - m_2 \left(\frac{(-q_1/q_{n+1} +a)^2}{1-a^2} + \sum_{i=2}^n \frac{q_i^2}{q_{n+1}^2}    \right)^{-1/2} 
\end{split}
\end{equation}
\normalsize
 The other $n-2$ first integrals are obtained as 
\begin{equation}
\label{eq: Ck_proj_ndim_hyp}
\hat{C}_k:= \sum_{2 \leq i < j \leq k} \hat{L}_{{i}j}^2, \qquad 3 \leq k \leq n,
\end{equation}
where $ \hat{L}_{{i}j} = q_i q_j' - q_j q_i'$ is the angular momentum.
}

{As reflection walls in $\mathcal{H}^n_S$, we consider confocal quadrics defined as follows.}
{We set the Kepler centers $Z_1,Z_2$ on $\mathcal{H}^n_S$ as the projections of $\tilde{Z}_1, \tilde{Z}_2$:
\[
Z_1 = \left(\frac{a}{\sqrt{1-a^2}}, 0, \cdots, 0, -\frac{1}{\sqrt{1-a^2}}\right), Z_2 = \left(-\frac{a}{\sqrt{1-a^2}}, 0, \cdots, 0, -\frac{1}{\sqrt{1-a^2}}\right).
\]
}

{Foci for quadrics in $\mathcal{H}^n_S$ are defined in the same way as in $\mathcal{H}_S$.}
\begin{defi}
	{A quadric with two foci on the $n$-dimensional hyperboloid model $(\mathcal{H}^n_S, \|\cdot \|_H)$ is called confocal if it is focused at the Kepler centers $Z_1,Z_2$.}
\end{defi}

{Such a quadric is centered at the Hooke center $Z_0 = (0,\cdots, 0, -1)$ and is, after normalization, given by 
\begin{equation}
	\left\{
	\begin{array}{l}
		\sum_{i=1}^{n} q_i^2 - q_{n+1}^2=-1, \quad q_{n+1}<0 \\
		\frac{q_1^2}{A^2}+\sum_{i=2}^{n} \frac{q_i^2}{B^2} - q_{n+1}^2 = 0, \quad A, B\in \R \setminus 0
	\end{array}
	\right.
\end{equation}
for a spheroid, and by
\begin{equation}
	\left\{
	\begin{array}{l}
		\sum_{i=1}^{n} q_i^2 - q_{n+1}^2=-1,\quad q_{n+1}<0 \\
		\frac{q_1^2}{A^2}-\sum_{i=2}^{n} \frac{q_i^2}{B^2} - q_{n+1}^2 = 0, \quad A, B\in \R \setminus 0
	\end{array}
	\right.
\end{equation}
for a circular hyperboloid of two sheets.}

{Note that the angular momentum {components} {$\hat{L}_{i j}, 2 \leq i,j \leq n$}, {and consequently also} the first integrals $\hat{C}_k$ for $3 \leq k \leq n$, are conserved under {the} reflections at {these} reflection walls due to the symmetry of the walls. Additionally, the first integral $\hat{E}_{sp}$ is also preserved under such reflections. {These facts follows directly from the lemma below. }

\begin{lemma}
	\label{lem: reflection_ndim_hyp}
	{Let $\tilde{\mathcal{B}}$ be a centered confocal quadric reflection wall in {the space} $(B^n, \| \cdot\|_{ia})$ with foci at $\tilde{F}_1= (a, 0, \cdots, 0)$ and $\tilde{F}_2 = (-a,0, \cdots, 0)$ in $B^n$. Let $\mathcal{B}$ be the quadric on {$\mathcal{H}^n_{S}$} given as the projection of $\tilde{\mathcal{B}}$. Then $\mathcal{B}$ is a centered confocal quadric having its foci at the projection of $\tilde{F}_1$ and $\tilde{F}_2$ on $\mathcal{H}^n_{S}$. Moreover, $\tilde{\mathcal{B}}$ and $\mathcal{B}$ are in projective correspondence.} 
\end{lemma} 
{This lemma directly follows from the {same argument} as in the three-dimensional case.}
{By combining the results in Lemma \ref{lem: Lagrange_proj_ndim_hyp} and Lemma \ref{lem: reflection_ndim_hyp}, we obtain the following integrability results for the Lagrange billiard defined on the hyperboloid {model} $\mathcal{H}^n$.}

\begin{theorem}
		{The mechanical billiard systems defined on the hyperboloid {model} $\mathcal{H}^n_S$ with the hyperbolic Lagrange problem 
		and with any finite combination of confocal symmetric centered 
		quadrics, with foci at the two Kepler centers as reflection walls, are integrable.}

	{As subcases, the mechanical billiard systems defined on $\mathcal{H}^n_{S}$}
	\begin{itemize}
		\item {with the hyperbolic two-center problem and with any finite combination of confocal quadrics with foci at the two centers as reflection walls;}
		\item {with the hyperbolic Kepler problem and with any finite combination of confocal quadrics with one of the foci at the Kepler center as reflection walls;}
		\item  {with the hyperbolic Hooke problem and with any finite combination of confocal quadrics centered at the Hooke center as reflection walls, }
	\end{itemize}	
	are integrable. {The first integrals are $E_{hyp}$ \eqref{eq: E_hyp_ndim}, $\hat{E}_{sp}$ \eqref{eq: E_sp_proj_ndim_hyp}, and $\hat{C}_k$ \eqref{eq: Ck_proj_ndim_hyp} for $3 \leq k \leq n$}.
\end{theorem}

{\bf Acknowledgement}	
A.T. and L.Z. are supported by DFG ZH 605/1-1, ZH 605/1-2. 


\hspace{-1cm}
\begin{tabular}{@{}l@{}}%
	Airi Takeuchi\\
	\textsc{University of Augsburg, Augsburg, Germany.}\\
	\textit{E-mail address}: \texttt{airi1.takeuchi@uni-a.de}
\end{tabular}
\vspace{10pt}

\hspace{-1cm}
\begin{tabular}{@{}l@{}}%
	Lei Zhao\\
	\textsc{University of Augsburg, Augsburg, Germany.}\\
	\textit{E-mail address}: \texttt{lei.zhao@math.uni-augsburg.de}
\end{tabular}

\end{document}